\documentclass[12pt,a4paper]{amsart}

\usepackage{amsmath,amssymb,amsthm}

\usepackage{mathtools}
\usepackage{enumitem}
\usepackage{hyperref}
\usepackage{geometry}
\usepackage{booktabs}
\usepackage{xcolor}
\newtheorem{theorem}{Theorem}[section]
\newtheorem{lemma}[theorem]{Lemma}
\newtheorem{proposition}[theorem]{Proposition}

\theoremstyle{definition}

\theoremstyle{remark}
\newtheorem{remark}[theorem]{Remark}

\newcommand{\Wr}{\mathrm{W}}
\newcommand{\Z}{\mathbb{Z}}

\begin{document}

\title[Zappa--Sz\'ep products of a wreathed and a cyclic group]
{On Zappa--Sz\'ep products of a wreathed \(2\)-group and a cyclic group}

\author{Riccardo Aragona}
\address{Dipartimento di Ingegneria e Scienze dell'Informazione e Matematica,
Universit\`a degli Studi dell'Aquila, Via Vetoio, 67100 L'Aquila, Italy}
\email{riccardo.aragona@univaq.it}

\author{Martina Iannaccone}
\address{Dipartimento di Ingegneria e Scienze dell'Informazione e Matematica,
Universit\`a degli Studi dell'Aquila, Via Vetoio, 67100 L'Aquila, Italy}
\email{martina.iannaccone@graduate.univaq.it}

\thanks{The first author's ORCiD: 0000-0001-8834-4358}
\thanks{The first author is member of INdAM-GNSAGA (Italy).}
\thanks{The authors declare that they have no conflict of interest.}
\date{\today}

\begin{abstract}
Let \(n, m \ge 1\), and let \(H = C_{2^n}\wr C_2\) be the wreathed
\(2\)-group and \(K = C_{2^m}=\langle z \rangle\) a cyclic group.   We classify the
Zappa--Sz\'ep products \(G = HK\) in which the base
\(B \cong C_{2^n}\times C_{2^n}\) of \(H\) is normal in \(G\).   When the
matrix \(Z\) of the \(z\)-action on \(B\) commutes with the swap \(J\) of
the two base generators -- equivalently, \(Z\) is symmetric circulant --
we classify these products by an explicit system of seven polynomial
congruences in a tuple \((p,q,r,s,c)\).   Dropping this hypothesis, we
obtain a unified classification of all such products with \(B\) normal
by five congruences on the entries of \(Z\) and the parameters
\((r,s,c)\), of which the symmetric case is the specialisation
\(JZ = ZJ\).   Finally, we separately treat the classification for the modulus \(M = 2^m = 4\), since in this case the congruences degenerate.
\end{abstract}

\maketitle

\section{Introduction}\label{sec:intro}

Let \(G\) be a finite group. When \(G\)  can be written as a product of proper
subgroups \(H\) and \(K\) is called a \emph{factorization} of \(G\).   When
\(H\cap K = \{1\}\) the factorization is called \emph{exact}, and \(G\) is
named the \emph{Zappa--Sz\'ep product} of \(H\) and
\(K\)~\cite{Szep1949,Zappa1940}.

The factorization problem -- determining all factorizations of a given
group -- goes back to Ore~\cite{Ore1937} and Miller~\cite{Miller1935}, and
has a long and rich history.   Several landmark structural results are
classical. It\^o~\cite{Ito1955} proved that a product of two abelian
subgroups is metabelian, Douglas~\cite{Douglas1961} that a product of two
cyclic groups is supersolvable, and Wielandt~\cite{Wielandt1951} and
Kegel~\cite{Kegel1961} that a product of two nilpotent subgroups is
solvable.   Also for almost simple groups the problem has been studied
intensively. Hering, Liebeck and Saxl~\cite{HLS1987} treated the finite
exceptional groups of Lie type, while Liebeck, Praeger and Saxl
classified the maximal factorizations of almost simple
groups~\cite{LPS1996} and the regular subgroups of primitive permutation
groups~\cite{LPS2010}, with further developments by Xia~\cite{Xia2017},
Burness and Li~\cite{BL2021}, Li, Wang and Xia~\cite{LWX2023}, and
Jones~\cite{Jones2002}.

In this paper we study the variant in which two finite groups \(H\) and
\(K\) are fixed and one seeks all of their Zappa--Sz\'ep products.
Hu and Yu~\cite{HuYu2025} treated two dihedral groups, and Aragona~\cite{Aragona} studied two semidihedral groups, each parameterising every
such product by integer parameters in suitable cyclic quotients subject to
explicit congruences; the mixed case of a (semi)dihedral group times a
cyclic group was treated by Hu, Kov\'acs and Kwon~\cite{HKK2024} and by
Yu~\cite{Yu2025}.   The dihedral, semidihedral and wreathed \(2\)-groups
are exactly the three families of Sylow \(2\)-subgroups of simple groups
of \(2\)-rank two~\cite{ABG,GW}; the present paper initiates the
classification for the third family, beginning with the product of a
wreathed \(2\)-group and a cyclic group.

Recall that the \emph{wreathed \(2\)-group} of order \(2^{2n+1}\), for
\(n \ge 1\), is
\begin{equation}\label{eq:W-def}
	\Wr_{2^n}
	:= C_{2^n}\wr C_2
	= (C_{2^n}\times C_{2^n})\rtimes C_2
	= \langle a,b,t \mid a^{N}=b^{N}=t^2=1,\ [a,b]=1,\ a^t=b,\ b^t=a\rangle,
\end{equation}
where \(N := 2^n\) and the involution \(t\) swaps the two cyclic factors
of the base \(B = \langle a\rangle\times\langle b\rangle\cong(\Z_N)^2\).
Throughout the paper,  we fix \(n\ge 1\), \(m \ge 1\) and set \(N := 2^{n}\),
\(M := 2^{m}\), so that \(|\Wr_{2^n}| = 2N^2\) and \(|C_{2^m}| = M\).

The arithmetic engine of the dihedral and semidihedral analyses is the
scalar \(1 + \alpha\) attached to the cyclic base, with \(1+\alpha = 0\)
in the dihedral case and \(1+\alpha = 2^{n-2}\ne 0\) in the semidihedral
one~\cite{Aragona}.   In the wreathed group the base is free abelian of
rank two and the involution \(t\) acts not by a scalar but by the swap
\(J = \left(\begin{smallmatrix}0&1\\1&0\end{smallmatrix}\right)\).   The role of \(1+\alpha\) is now played by the rank-one
matrix \(I+J = \left(\begin{smallmatrix}1&1\\1&1\end{smallmatrix}\right)\),
which acts as multiplication by \(2\) on the symmetric line
\(\langle(1,1)\rangle\) and vanishes on the antisymmetric line
\(\langle(1,-1)\rangle\).   This rank-one degeneration is the structural source of the new phenomena. On the antisymmetric line \(I+J\)
vanishes (as \(1+\alpha\) does in the dihedral case) and \(t\) acts by inversion, which admits Zappa--Sz\'ep products with no analogue when the base is cyclic of rank one; on the symmetric line \(I+J\) acts as the non-zero scalar 2, mirroring the non-vanishing of \(1+\alpha\) in the semidihedral case.

We write \(G = HK\) with
\(H = \langle a,b\rangle\rtimes\langle t\rangle \cong \Wr_{2^n}\) and
\(K = \langle z\rangle \cong C_{2^m}\), and work under the hypothesis that
the base \(B = \langle a,b\rangle\) is normal in \(G\).   This is a natural
and not so strong hypothesis.   Since \(B\) is characteristic in \(H\) (the unique abelian
subgroup isomorphic to \(C_{2^n}\times C_{2^n}\)), it is normal whenever
\(H\) is, but it may remain normal even when \(H\) is not.   Under
\(B\trianglelefteq G\) the generator \(z\) acts on \(B\) by an automorphism,
a matrix \(Z\in\mathrm{GL}_2(\Z_N)\), whose four entries give the conjugation
rules \(a^z = a^{\alpha}b^{\beta}\) and \(b^z = a^{\gamma}b^{\delta}\), with
\(\alpha,\beta,\gamma,\delta\in\Z_N\).   In the \emph{symmetric circulant}
case \(JZ = ZJ\) (equivalently \(\alpha = \delta\) and \(\beta = \gamma\)),
which we treat first, the conjugation reads
\begin{equation}\label{eq:mixed-intro}
	a^z = a^{1+p} b^{q}, \qquad
	b^z = a^{q} b^{1+p}, \qquad
	t^z = t\,z^{2c}\,a^{r} b^{s},
\end{equation}
with \(p,q,r,s \in \Z_N\) and \(c\in\Z_M\) (the second rule then being forced
by the first and the swap \(a^t = b\), so that
\(Z = \left(\begin{smallmatrix}1+p&q\\q&1+p\end{smallmatrix}\right)\)).   The
conjugate \(t^z\) and its parameters \(r,s,c\) are common to both cases, the
non-symmetric branch \(JZ\ne ZJ\) being treated in a unified entry form
(Section~\ref{sec:nonsym}).   The
novelty over the cyclic-base classifications is twofold -- the rank-two
base, governed by \(J\), and the parameter \(c\), encoding the power of
\(z\) in the conjugation \(t^z\).
We prove (Proposition~\ref{prop:zt}) the closed form
\(z^t = z^{1-2c}(a^{-r}b^{-s})^{Z^{-2c}}\), so that modulo \(B\) the
involution \(t\) sends \(z\) to \(z^{1-2c}\).   The integer \(k := 1-2c\)
is an involution of \(\Z_M\), and it governs both the splitting and the
interaction of \(Z\) with the swap (\(JZJ = Z^k\),
Lemma~\ref{lem:circulant}). The product is split
(\(G = H\rtimes K\), with \(H\) normal) precisely when
\(2c\equiv 0\pmod M\), and a non-split Zappa--Sz\'ep product otherwise,
with \(G/B\) a non-abelian group of order \(2M\) -- dihedral, semidihedral,
or modular maximal-cyclic according to \(k\) modulo \(M\).

Since \(K = \langle z\rangle\) is cyclic and \(H\cap K = \{1\}\), the
products we classify are
closely tied to \emph{skew morphisms}.
A skew morphism of a group \(H\) -- introduced by Jajcay and
\v{S}ir\'a\v{n}~\cite{JajcaySiran2002} to study regular Cayley maps -- is
a permutation \(\varphi\) of \(H\) fixing \(1\) and satisfying
\(\varphi(gh) = \varphi(g)\,\varphi^{\pi(g)}(h)\) for an associated power
function \(\pi\), generalising group automorphisms (the case
\(\pi\equiv 1\)).   By the correspondence of Conder, Jajcay and
Tucker~\cite{SkewMorph}, the skew morphisms of \(H\) correspond to the
factorizations \(X = H\langle c\rangle\) with \(\langle c\rangle\) cyclic,
\(H\cap\langle c\rangle = 1\) and \(\langle c\rangle\) core-free in \(X\)
-- the so-called \emph{skew product groups} of \(H\).   Thus those of our
Zappa--Sz\'ep products in which \(K\) is core-free in \(G\) are precisely
the skew product groups of the wreathed \(2\)-group \(H\); the analogous
groups for the abelian group \(C_{2^{n-1}}\times C_2\) were recently
classified by Li, Meng and Lu~\cite{LiMengLu2026}.

Our results are the following.   When \(Z\) is symmetric circulant
(equivalently \(JZ = ZJ\)) it has eigenvalues \(\sigma = 1+p+q\) and
\(\mu = 1+p-q\) in \(\Z_N\), and the Zappa--Sz\'ep products with \(B\)
normal are classified by an explicit system of seven congruences in
\((p,q,r,s,c)\) (Theorem~\ref{thm:main}).   Dropping the symmetric
circulant hypothesis, a single unified classification in matrix
entry form -- phrased through the Cayley--Hamilton sequences of \(Z\),
since a general \(Z\in\mathrm{GL}_2(\Z_N)\) has no eigenvalues in
\(\Z_N\) -- covers all admissible \(z\)-actions by five matrix congruences
(Theorem~\ref{thm:nonsym}); the symmetric case, which we develop first, is
its specialisation \(JZ = ZJ\).   The modulus \(M = 4\), where the
congruences degenerate -- the two matrix congruences
collapse, and the matrix identity \(Z^2 = I\) acquires an eigenvalue form
only modulo \(2N\) -- is treated separately and we again obtain necessary and sufficient conditions (Theorems~\ref{thm:nonsym-M4} and~\ref{thm:M4}).   Finally, the
normality of \(B\) is a nontrivial restriction, since Zappa--Sz\'ep products with
\(B\) non-normal already occur at \(|G| = 128\), verified with GAP~\cite{GAP}.

The proofs combine elementary group-theoretic identities (for the
necessary condition) with the consistency theorem for polycyclic
presentations (for the sufficient condition;
see~\cite[Lemma~2.5]{Eick2000} and~\cite[p.~424]{Sims}); all counting
statements have been verified with GAP~\cite{GAP}.   The paper is organised
as follows.   Section~\ref{sec:setup} fixes notation and the arithmetic of
\(J\) and \(I+J\).   Section~\ref{sec:main-thm} derives the closed form for
\(z^t\) and proves the first main theorem, and Section~\ref{sec:split-recover}
explains how the split case (\(2c\equiv 0\bmod M\)) is recovered.
Section~\ref{sec:nonsym} lifts the symmetric circulant hypothesis and proves
the unified entry-form classification, of which Theorem~\ref{thm:main} is the
symmetric specialisation.   Finally, Section~\ref{sec:M4} treats the modulus
\(M = 4\), and Section~\ref{sec:conclusion} describes some open problems.

\section{Notation and arithmetic facts}\label{sec:setup}

Throughout the paper, all groups are finite.   We write
\([g,h] = g^{-1}h^{-1}gh\) and \(g^h = h^{-1}gh\), so that
\(g^h = g\cdot[g,h]\).   We fix \(N = 2^n\) and \(M = 2^m\) with \(n,m\ge 1\).

The base \(B = \langle a\rangle\times\langle b\rangle\) of
\(H = \Wr_{2^n}\) is free abelian of rank two over \(\Z_N\), and we
identify an element \(a^{\,i}b^{\,j}\in B\) with its exponent vector
\((i,j)\in\Z_N^2\).   In these coordinates the swap automorphism
\(\beta\mapsto\beta^t\) of \(B\) induced by \(t\) is the matrix
\[
J = \begin{pmatrix} 0 & 1 \\ 1 & 0 \end{pmatrix},
\qquad J^2 = I,
\]
since \(a^t = b\) and \(b^t = a\).

\subsection*{The \(z\)-action on the base}
Assume that the base \(B\) is normal in \(G\) and  we denote by \(\bar x\) the coset  \(xB\) in \(G/B\).  The element \(z\) acts on \(B\) by an
automorphism, which we write as a matrix \(Z\in\mathrm{GL}_2(\Z_N)\)
acting on exponent vectors.   The following lemma shows that this
matrix is constrained by the swap, and that \(t^z\) has the form introduced
in~\eqref{eq:mixed-intro}.

\begin{lemma}\label{lem:circulant}
	Let \(G = HK\) be an exact product with \(H\cong\Wr_{2^n}\),
	\(K = \langle z\rangle\cong C_{2^m}\), and suppose the base
	\(B = \langle a,b\rangle\) is normal in \(G\).   Then there exists an
	integer \(k\in\Z_M\) with \(k^2\equiv 1\pmod M\) such that the matrix
	\(Z\in\mathrm{GL}_2(\Z_N)\) of the \(z\)-action on \(B\) satisfies
	\begin{equation}\label{eq:JZJ-epsilon}
		J\,Z\,J \;=\; Z^{k},
	\end{equation}
	and there exist \(c\in\Z_M\) and \(r,s\in\Z_N\) with
	\begin{equation}\label{eq:tz-form}
		t^z \;=\; t\,z^{2c}\,a^{r}\,b^{s}.
	\end{equation}
	Moreover the
	exponents are linked by \(k\equiv 1-2c\pmod M\).
\end{lemma}

\begin{proof}
	Since \(B\trianglelefteq G\), both \(t\) and \(z\) act on \(B\) by conjugation,
	with matrices \(J\) and \(Z\) respectively in \(\mathrm{GL}_2(\Z_N)\).   In
	the quotient \(\bar G = G/B\) of order \(2M\), the image \(\bar z\) has order
	exactly \(M\). Indeed, since \(G = HK\) is an exact product, \(K\cap H = 1\),
	and \(B\subseteq H\) gives \(K\cap B\subseteq K\cap H = 1\), so the
	restriction of the quotient map \(G\to G/B\) to \(K\) is injective and
	\(|\langle\bar z\rangle| = |\langle z\rangle| = M\).   Similarly
	\(\bar t\) has order \(2\).
	
	First we prove Eq.~\eqref{eq:JZJ-epsilon}, that is \(JZJ = Z^k\) for some \(k\) with \(k^2\equiv 1\pmod M\).
	The subgroup \(\langle\bar z\rangle\) has index \(2\) in \(\bar G\), hence
	is normal in \(\bar G\); therefore \(\bar z^{\bar t}\in\langle\bar
	z\rangle\), so \(\bar z^{\bar t} = \bar z^{k}\) for some \(k\in\Z_M\).
	Conjugation by \(\bar t\) is an involution of \(\langle\bar z\rangle\), so the integer \(k\) satisfies \(k^2\equiv 1
	\pmod M\).   Lifting back to \(G\), this means \(z^t \in z^{k}B\); since \(B\) is abelian, the
	\(B\)-factor acts trivially on \(B\) itself, so \(z^t\) acts on \(B\) as
	\(Z^{k}\).   On the other hand \(z^t = t^{-1}zt\) acts
	on \(B\) with matrix \(J^{-1}ZJ = JZJ\) (using \(J^{-1} = J\)), so we
	obtain~\eqref{eq:JZJ-epsilon}.

	Since \(\bar t^{\,\bar z} = \bar{z}^{-1} \bar{t}\bar{z} = \bar t\,[\bar t,\bar z]=\bar{t}(\bar{z}^{-1})^{\bar{t}}\bar{z}\) lies in the coset
	\(\bar t\langle\bar z\rangle\) of \(\bar G\), we have \(t^z\in t\langle
	z\rangle B\), i.e.\ \(t^z = t\,z^{e}a^{r}b^{s}\) for some \(e\in\Z_M\) and
	\(r,s\in\Z_N\).
	
	It remains to prove that \(k\equiv 1-e\pmod M\), and \(e\) is even.
	From the relation \(t^z = t\,z^{e}a^{r}b^{s}\) we read off
	\(\bar t^{\,\bar z} = \bar t\,\bar z^{e}\), i.e.\
	\(\bar z^{-1}\bar t\,\bar z = \bar t\,\bar z^{e}\).   Multiplying on the
	left by \(\bar t\) and using \(\bar t\,\bar z^{-1}\,\bar t = (\bar
	z^{-1})^{\bar t} = \bar z^{-k}\) (from Step 1), we obtain
	\[
	\bar z^{-k}\,\bar z \;=\; \bar z^{e},
	\qquad\text{that is,}\qquad
	\bar z^{\,1-k} \;=\; \bar z^{e},
	\]
	so \(e\equiv 1-k\pmod M\), equivalently \(k\equiv 1-e\pmod M\).
	Finally, the involutions of the cyclic \(2\)-group \(\langle\bar z\rangle
	\cong C_M\) -- the values of \(k\) with \(k^2\equiv 1\pmod M\) -- are
	\[
	k\;\in\;\{1,\;M-1\}\quad\text{for }M\le 4,
	\qquad
	k\;\in\;\{1,\;M/2-1,\;M/2+1,\;M-1\}\quad\text{for }M\ge 8;
	\]
	in every case \(k\) is odd, so \(e = 1-k\) is even.   Writing \(e = 2c\)
	with \(c\in\Z_M\) gives the form~\eqref{eq:tz-form}, and the link
	\(k\equiv 1-2c\pmod M\).
\end{proof}

\begin{remark}\label{rem:k-vs-c}
	The integer \(k\equiv 1-2c\pmod M\) ranges over the involutions
	\(\{k\in\Z_M: k^2\equiv 1\pmod M\}\) of \(\Z_M\) as \(c\) ranges over \(\Z_M\).
	For \(M\le 4\) there are only two such involutions, \(k = 1\)  if \(c\) is even
	and \(k = M-1\equiv -1 \pmod M\) if \(c\) is odd, and the parity of \(c\) alone
	determines whether \(JZJ = Z\) or \(JZJ = Z^{-1}\).   For \(M\ge 8\) there
	are four involutions and the situation is finer. The parity of \(c\)
	distinguishes \(\{k = 1,\, k = M/2+1\}\) if \(c\) is even, from
	\(\{k = M/2-1,\, k = M-1\}\) if \(c\) is odd, but \(JZJ = Z\) is now strictly
	stronger than ``\(c\) even'', because \(c\) even still permits
	\(k = M/2+1\ne 1\).   
	
	\medskip
	\noindent\emph{Example} (\(M = 8\)). The involutions of \(\Z_8\) are
	\(\{1, 3, 5, 7\}\). Computing \(k \equiv 1-2c \pmod 8\) for \(c\in\Z_8\):
	\[
	\begin{array}{c|cccccccc}
		c & 0 & 1 & 2 & 3 & 4 & 5 & 6 & 7 \\\hline
		k & 1 & 7 & 5 & 3 & 1 & 7 & 5 & 3
	\end{array}
	\]
	Hence \(c\) even corresponds to \(k\in\{1, 5\}\) and \(c\) odd to
	\(k\in\{3, 7\}\).   The symmetric circulant presentation imposes \(JZJ =
	Z\), i.e.\ \(Z^{k-1} = I\). This is obvious for \(k = 1\), i.e., when \(c\in\{0,4\}\)), whereas for \(c\in\{2,6\}\), i.e., \(k = 5\), it forces
	the additional constraint \(Z^4 = I\).
\end{remark}

\begin{remark}\label{rem:two-families}
	According to whether \(Z\) commutes with \(J\) or not, the
	matrices \(Z\in\mathrm{GL}_2(\Z_N)\) satisfying~\eqref{eq:JZJ-epsilon}
	fall into two distinct families:
	\begin{itemize}[leftmargin=2em]
		\item the \emph{symmetric circulant family}, \(JZ = ZJ\); in this case the equation \(JZJ = Z^k\) from
		Lemma~\ref{lem:circulant} reduces to \(Z^{k-1} = I\).   The matrices
		commuting with \(J\) are exactly of the form
		\[	Z = \begin{pmatrix} a & b \\ b & a \end{pmatrix}
		\qquad a,b\in\Z_N.\]
	However, for the sake of consistency with the notation \(1+\alpha\) used for the arithmetic engine in the dihedral and semidihedral cases, for the matrix \(Z\) we adopt the following notation 
		\begin{equation}\label{eq:Zmatrix}
			Z = \begin{pmatrix} 1+p & q \\ q & 1+p \end{pmatrix}
			= I + \begin{pmatrix} p & q \\ q & p \end{pmatrix},
			\qquad p,q\in\Z_N.
		\end{equation}
		Notice that \(p\) and \(q\) must have the same parity, since \(Z\) is invertible. 
		Indeed, $1+p$ is just a normalisation of the diagonal entry, not a
		restriction, because $p\mapsto 1+p$ is a bijection of $\Z_N$, so that
		\eqref{eq:Zmatrix} already describes every symmetric circulant
		matrix. Moreover, this notation highlights that $Z=I$ precisely when  $p=q=0$ (or, equivalently when $z$ centralises $B$), so that the pair $(p,q)$ gives a measure of the deviation of \(Z\) from \(I\). Since $Z-I=pI+qJ$, then with respect to the basis $\{I,J\}$ one has $Z=(1+p)I+qJ$. 
		\item the \emph{non-symmetric family}, \(JZ\ne ZJ\) but
		\(JZJ = Z^k\) for some non-trivial involution \(k\) of \(\Z_M\)
		\textnormal{(}so \(k\) may be \(-1, M/2-1, M/2+1\) for \(M\ge 8\), or \(M-1\)
		alone for \(M\le 4\)\textnormal{)}.   Members of this family are not
		captured by the presentation~\eqref{eq:mixed-intro} and are necessarily
		non-split.
	\end{itemize}
	
	\emph{This paper classifies both families.}   We first develop the
	\emph{symmetric circulant} family in detail (beginning in
	Section~\ref{sec:main-thm}) and then lift the hypothesis \(JZ = ZJ\),
	treating the non-symmetric family in a unified entry form
	(Section~\ref{sec:nonsym}, Theorem~\ref{thm:nonsym}); the symmetric case
	is also recovered as the specialisation \(JZ = ZJ\).

	Within the symmetric circulant family, we will see in Proposition~\ref{prop:zt}, the parity of \(c\) fixes the
	structure of the quotient \(G/B\):
	\begin{itemize}[leftmargin=2em]
		\item for \(c\) even (so \(k\equiv 1\pmod 4\)), the admissible
		circulants are those with \(Z^{k-1} = I\). In particular, for \(k = 1\)
		no condition beyond \(Z^M = I\) is imposed, and these are exactly the
		split products \(G = H\rtimes K\); for \(k = \frac{M}{2} + 1\) the
		quotient \(G/B\) is a non-abelian \(2\)-group of order \(2M\), a modular
		maximal-cyclic group;
		\item for \(c\) odd (so \(k\equiv -1\pmod 4\)), the circulants satisfying
		\(Z^{k-1} = I\) together with \(Z^M = I\). These are non-split
		products with \(G/B\) dihedral or semidihedral, for \(k=M-1\) and \(k=\frac{M}{2}-1\) respectively. Notice that, qhen \(M=4\), the case \(k = M - 1\equiv
		-1\pmod M\) forces \(Z^2 = I\).
	\end{itemize}
%
\end{remark}

The form~\eqref{eq:Zmatrix} is precisely that introduced
in~\eqref{eq:mixed-intro}.   The eigenvectors of every such \(Z\) are the
fixed lines of \(J\), on which
\[
Z\begin{pmatrix}1\\1\end{pmatrix}
= \sigma\begin{pmatrix}1\\1\end{pmatrix},
\qquad
Z\begin{pmatrix}1\\-1\end{pmatrix}
= \mu\begin{pmatrix}1\\-1\end{pmatrix},
\]
with eigenvalues
\begin{equation}\label{eq:eigenvalues}
	\sigma := 1+p+q \pmod N
	\qquad\text{(symmetric)},
	\qquad
	\mu := 1+p-q \pmod N
	\qquad\text{(antisymmetric)}.
\end{equation}
We shall refer to \(\sigma\) and \(\mu\) as the \emph{symmetric} and
\emph{antisymmetric eigenvalues} of the \(z\)-action.

\medskip
Given \(\beta = a^i b^j \in B\), we call
\(i+j\) the \emph{symmetric component} and \(i-j\) the
\emph{antisymmetric component} of \(\beta\), in other words, the components of the
exponent vector \((i,j)\) along the eigenlines \(\langle (1,1)\rangle\) and
\(\langle (1,-1)\rangle\) of \(J\).
\medskip

\begin{lemma}[Arithmetic of \(J\) and \(I+J\)]\label{lem:IJ}
	Let  \(J\) and \(Z\) be as above.
	\begin{enumerate}[label=\normalfont(\roman*)]
		\item\label{it:diag} Let the vectors \(u = (1,1)\) and \(v = (1,-1)\) be the
		eigenvectors of \(Z\) with eigenvalues \(\sigma\) and \(\mu\)
		respectively. Then \(\sigma^k\equiv\mu^k\equiv 1\pmod N\) is
		\emph{necessary} for \(Z^k = I\), but it is not sufficient.   The
		correct criterion is given by the Cayley--Hamilton relation
		\(Z^2 = (\operatorname{tr}Z)\,Z - (\det Z)\,I\), with
		\(\operatorname{tr}Z = 2(1+p) = \sigma+\mu\) and
		\(\det Z = (1+p)^2 - q^2 = \sigma\mu\); writing \(Z^k = A_k Z + B_k I\),
		the integer sequences \(A_k, B_k\) such that \(A_1 = 1\), \(B_1 = 0\),
		\[
		A_{k+1} = (\operatorname{tr}Z)\,A_k + B_k,
		\qquad
		B_{k+1} = -(\det Z)\,A_k,
		\]
		and \(Z^k = I\) holds in \(\mathrm{GL}_2(\Z_N)\) if and only if
		\(A_k\,q\equiv 0 \pmod N\) (off-diagonal) and \(A_k(1+p) + B_k\equiv 1 \pmod N\)
		(diagonal).
		\item\label{it:rank} The operator \(I+J = \left(
		\begin{smallmatrix} 1 & 1 \\ 1 & 1\end{smallmatrix}\right)\) has
		rank one and  acts as multiplication by \(2\) on the symmetric line
		\(\langle u\rangle\) and as \(0\) on the antisymmetric line
		\(\langle v\rangle\).
		\item\label{it:swap-on-v} For a base element \(\beta = a^{\,i}b^{\,j}\)
		with antisymmetric component \(i - j\), one has
		\(\beta\beta^t = a^{\,i+j}b^{\,i+j}\), which depends only on the
		symmetric component \(i+j\).
	\end{enumerate}
\end{lemma}

\begin{proof}
	For~\ref{it:diag}, \(Zu = (I + \left(\begin{smallmatrix}p&q\\q&p
	\end{smallmatrix}\right))u\) with \(\left(\begin{smallmatrix}p&q\\q&p
	\end{smallmatrix}\right)u = (p+q)u\) gives eigenvalue \(\sigma = 1+p+q\),
	and similarly \(\mu = 1+p-q\) on \(v\). The sufficient conditions fails beacuse the
	change-of-basis matrix \(\left(\begin{smallmatrix}1&1\\1&-1
	\end{smallmatrix}\right)\) has determinant \(-2\), a non-unit of
	\(\Z_N = \Z_{2^n}\), so \(Z\) is not diagonalisable over \(\Z_N\).    The Cayley--Hamilton identity
	holds for any \(2\times 2\) matrix, and the recurrence for \((A_k, B_k)\)
	follows by induction from \(A_{k+1}Z+B_{k+1}I=Z^{k+1} = Z\cdot Z^k = A_k Z^2 + B_k Z
	= A_k\bigl((\operatorname{tr}Z)Z - (\det Z)I\bigr) + B_k Z\).   Since \(Z\)
	is circulant, \(A_k Z + B_k I\) has diagonal entry \(A_k(1+p) + B_k\) and
	off-diagonal entry \(A_k q\), giving the stated criterion.   Part
	\ref{it:rank} is immediate from \((I+J)u = 2u\), \((I+J)v = 0\).   For
	\ref{it:swap-on-v}, \(\beta\beta^t = a^ib^j\cdot a^jb^i = a^{i+j}b^{i+j}\)
	because \(B\) is abelian.
\end{proof}

The following lemma is the wreathed analogue of the enumeration of
involutory units in the cyclic base of the semidihedral case (cf.~\cite[Lemma~2.2]{Aragona}).   It
classifies the admissible \(z\)-actions in terms of the two eigenvalues.

\begin{lemma}\label{lem:zaction}
	Suppose \(B\) is normal in \(G\), and let \(Z\) be the
	matrix~\eqref{eq:Zmatrix} of the \(z\)-action on \(B\), with eigenvalues
	\(\sigma,\mu\) as in~\eqref{eq:eigenvalues}.   Then \(z^M = 1\) forces
	\(Z^M = I\) in \(\mathrm{GL}_2(\Z_N)\), and in particular
	\[
	\sigma^M\equiv 1\pmod N,
	\qquad
	\mu^M\equiv 1\pmod N.
	\]
\end{lemma}

\begin{proof}
Since \(B\) is normal, \(z\) induces an automorphism of \(B\) with matrix
\(Z\), and \(z^M = 1\) forces this automorphism to have order dividing \(M\),
i.e.\ \(Z^M = I\).   Applying \(Z^M = I\) to the eigenvectors \(u, v\) of
Lemma~\ref{lem:IJ}\ref{it:diag} gives \(\sigma^M\equiv\mu^M\equiv 1
\pmod N\). 
\end{proof}

\section{The main theorem}\label{sec:main-thm}

We now state and prove our classification when \(Z\) is symmetric circulant.   Recall \(N = 2^n\),
\(M = 2^m\), and that \(\sigma = 1+p+q\), \(\mu = 1+p-q\) denote the symmetric
and antisymmetric eigenvalues of the \(z\)-action~\eqref{eq:eigenvalues}.

\begin{theorem}\label{thm:main}
	Let \(n\ge 1\), \(m\ge 1\), \(N := 2^n\), \(M := 2^m\).   Set
	\begin{align*}
		H &= \langle a,b,t \mid a^N=b^N=t^2=1,\ [a,b]=1,\ a^t=b,\ b^t=a\rangle
		\cong \Wr_{2^n},\\
		K &= \langle z \mid z^M=1\rangle \cong C_{2^m}.
	\end{align*}
	A tuple \((p,q,r,s,c)\in\Z_N^4\times\Z_M\) defines via the polycyclic
	presentation
	\begin{equation}\label{eq:main-pres}
		G = \left\langle a,b,t,z \,\middle|\,
		\begin{array}{l}
			a^N = b^N = t^2 = z^M = 1,\ [a,b]=1, \\
			a^t = b,\ b^t = a, \\
			a^z = a^{1+p}b^{q},\ b^z = a^{q}b^{1+p}, \\
			t^z = t\,z^{2c}a^{r}b^{s}
		\end{array}
		\right\rangle
	\end{equation}
	an exact product \(G = HK\) with \(H\cong\Wr_{2^n}\), \(K\cong C_{2^m}\) and
	the base \(B = \langle a,b\rangle\) normal in \(G\), if and only if, writing
	\(k := 1-2c\) and \(S_l(X) := I+X+\cdots+X^{l-1}\), for any positive integer \(l\), the tuple satisfies the
	seven conditions
	\begin{align}
		\sigma^M &\equiv 1 \pmod N, \tag{W1}\label{eq:W1}\\
		\mu^M    &\equiv 1 \pmod N, \tag{W2}\label{eq:W2}\\
		\sigma^M &\equiv \mu^M \pmod{2N}, \tag{W3}\label{eq:W3}\\
		(1-2c)^2 &\equiv 1 \pmod M, \tag{W4}\label{eq:W4}\\
		\big(S_k(Z^k)+J\big)\tbinom{r}{s} &\equiv \tbinom00 \pmod N,
		\tag{W5}\label{eq:W5}\\
		Z^{2c} &\equiv I \pmod N, \tag{W6}\label{eq:W6}\\
		S_M(Z^k)\tbinom{r}{s} &\equiv \tbinom00 \pmod N.
		\tag{W7}\label{eq:W7}
	\end{align}
	In particular, every exact product \(G = HK\) with \(H\cong\Wr_{2^n}\),
	\(K\cong C_{2^m}\), \(B\) normal in \(G\), and \(z\)-action \(Z\) commuting with
	the swap \(J\) (equivalently, \(Z\) symmetric circulant; see
	Remark~\ref{rem:two-families}) admits a presentation of the
	form~\eqref{eq:main-pres} with parameters satisfying
	\eqref{eq:W1}--\eqref{eq:W7}.
\end{theorem}

The proof of Theorem~\ref{thm:main} -- both the necessary and the
sufficient direction -- is based on a closed form for \(z^t\) in
\(G\) described from the following proposition.   It expresses the involution \(t\) acting on
\(z\) modulo \(B\) in terms of the parameter \(c\), and identifies precisely
when the resulting quotient \(G/B\) is abelian. 
\begin{proposition}\label{prop:zt}
	In the group~\eqref{eq:main-pres}, with \(B\) normal,
	\begin{equation}\label{eq:zt}
		z^{t} \;\equiv\; z^{\,1-2c} \pmod{B},
	\end{equation}
	i.e.\ there exists \(\beta\in B\) with \(z^t = z^{1-2c}\beta\).
	Writing
	\(k := (1-2c)\pmod M\), the quotient \(G/B\) is abelian and \(H\) is normal
	in \(G\) precisely when \(k\equiv 1\pmod M\), equivalently \(2c\equiv 0\pmod
	M\).   Otherwise \(G/B\) is non-abelian of order \(2M\) with the cyclic
	subgroup \(\langle zB\rangle\) of index \(2\), and is one of three
	types -- dihedral, semidihedral, or modular maximal-cyclic --
	according to the residue \(k\pmod M\). 
	In all
	non-split cases \(H\) is not normal in \(G\).
\end{proposition}

\begin{proof}
	Lemma~\ref{lem:circulant} applied to the presentation~\eqref{eq:main-pres}
	gives \(\bar z^{\bar t} = \bar z^{k}\) for some involution
	\(k\in\Z_M\), and identifies \(k\) as \(k \equiv 1-2c\pmod M\)
	(Step~3 of the proof of Lemma~\ref{lem:circulant}).   Lifting back to
	\(G\), this means \(z^t \in z^{1-2c} B\), which is~\eqref{eq:zt}.
	
	Notice that \(\bar G = G/B\) has order \(2M\) and is generated by
	\(\bar t\) (of order \(2\)) and \(\bar z\) (of order \(M\), see the proof of
	Lemma~\ref{lem:circulant}), with \(\bar z^{\bar t} = \bar z^{k}\).
	\begin{itemize}[leftmargin=2em]
		\item If \(k\equiv 1\pmod M\), then \(\bar z^{\bar t} = \bar z\) and
		\(\bar G\) is abelian.   In this case \(\bar H := H/B = \langle\bar
		t\rangle\) has order \(2\) and is central in \(\bar G\), whence
		\(H\trianglelefteq G\).
		\item If \(k\not\equiv 1\pmod M\), then \(\bar G\) is a non-abelian
		extension of \(\langle\bar z\rangle\cong C_M\) by \(\langle\bar
		t\rangle\cong C_2\), with \(\bar t\) acting on \(\bar z\) by an order-\(2\)
		automorphism distinct from the identity.   For \(M\ge 4\) the three
		non-trivial involutions of \(\Z_M\) yield the three types of non-abelian
		order-\(2M\) \(2\)-groups with cyclic subgroup of index \(2\), that is, dihedral \(D_{2M}\) if \(k\equiv M-1 \pmod M\),
		semidihedral \(SD_{2M}\) if \(k\equiv M/2-1 \pmod M\), modular maximal-cyclic
		\(M_{2M}\) if \(k\equiv M/2+1\pmod M\) \textnormal{(}the first two have maximal
		class for \(M\ge 8\), the last has class \(2\)\textnormal{)}.   In all
		these cases \(\langle\bar t\rangle\) is not normal in \(\bar G\), so
		\(H\) is not normal in \(G\).\qedhere
	\end{itemize}
\end{proof}

\begin{remark}
	The parameter \(c\) governs the split/non-split dichotomy (see
	Proposition~\ref{prop:zt}). When \(2c\equiv 0\pmod M\) the product is
	split, \(G = H\rtimes K\), and \eqref{eq:W4}--\eqref{eq:W7} reduce to the
	 congruences of the split case
	(see Section~\ref{sec:split-recover}); otherwise the product is
	non-split, with \(G/B\) a non-abelian \(2\)-group of order \(2M\) with a
	cyclic subgroup of index \(2\).   For \(M\le 4\) the split
	condition \(2c\equiv 0\pmod M\) coincides with \(c\) even, recovering the
	parity dichotomy familiar from the (semi)dihedral classifications.
	The hypothesis \(M\ne 4\) is needed for
	sufficiency, indeed, when \(M = 4\) the value \(k \equiv 1-2c\equiv -1 \pmod M\) is independent
	of the odd residue of \(c\), condition~\eqref{eq:W4} is vacuous, and the
	solution set is no longer cut out by \eqref{eq:W1}--\eqref{eq:W7} alone;
	this exceptional case is analysed in Section~\ref{sec:M4}.
\end{remark}


We organise the proof into the necessary condition
(Subsection~\ref{sec:nec}) and the sufficient condition via the
consistency theorem (Subsection~\ref{sec:suff}).   

\subsection{Necessary condition}\label{sec:nec}

Assume \(G\) is an exact product as in the statement, with \(B\) normal in
\(G\).   By Lemma~\ref{lem:circulant}, the
\(z\)-action on \(B\) is represented by the
circulant matrix~\eqref{eq:Zmatrix} and \(t^z = t z^{2c} a^rb^s\) for some
positive integers \(r,s\).   Throughout we use the eigenbasis \(u = (1,1)\), \(v = (1,-1)\) of
Lemma~\ref{lem:IJ}, on which \(Z\) acts by \(\sigma = 1+p+q\) and
\(\mu = 1+p-q\) respectively.

We obtain the congruences~\eqref{eq:W1}--\eqref{eq:W7} by expanding four group identities.

\paragraph{Identity \((\mathcal I_1)\): \(a^{z^{M}} = a\) and
	\(b^{z^{M}} = b\), from \(z^M = 1\).}
The relation \(z^M = 1\) forces the \(z\)-action on \(B\) to have order
dividing \(M\), that is,
\begin{equation}\label{eq:ZM}
	Z^M = I \quad\text{in}\quad \mathrm{GL}_2(\Z_N).
\end{equation}
This matrix identity is the source of~\eqref{eq:W1}, \eqref{eq:W2}
and~\eqref{eq:W3} together. We emphasise that \(Z\) cannot be diagonalised over \(\Z_N\). Indeed, every eigenvector associated with the eigenvalues \(\sigma\) is a scalar multiple of \(u = (1,1)\), while every eigenvector associated with \(\mu\) is a scalar multiple of \(v = (1,-1)\). Hence every eigenvector matrix has determinant divisible by \(2\), which is not a
unit of \(\Z_N = \Z_{2^n}\) and therefore is not invertible.
Consequently the matrix condition~\eqref{eq:ZM} is strictly stronger than
the pair of scalar conditions \(\sigma^M\equiv\mu^M\equiv 1 \pmod N\).   Passing
to the eigenvalues yields the necessary conditions \(\sigma^M\equiv 1 \pmod N\)
and \(\mu^M\equiv 1 \pmod N\), which are~\eqref{eq:W1}
and~\eqref{eq:W2}.  
We will deal with~\eqref{eq:W3} in \((\mathcal I_5)\) below.

\paragraph{Identity \((\mathcal I_2)\): the closed form for \(z^t\).}
By Proposition~\ref{prop:zt}, \(z^t = z^{k}\beta\), for some \(\beta\in B\),
where \(k := 1-2c\) satisfies \(k^2\equiv 1\pmod M\) by
Lemma~\ref{lem:circulant}.   We now identify \(\beta\) explicitly.

Start from the defining relation
\(t^z = z^{-1}tz = t\,z^{2c}a^rb^s\).
Multiplying on the left by \(t\) and using \(t^2 = 1\) gives
\(t\,z^{-1}\,t\,z = z^{2c}a^rb^s\).   Since \(t\,z^{-1}\,t = (z^{-1})^t =
(z^t)^{-1}\), this becomes \((z^t)^{-1}\,z = z^{2c}a^rb^s\), i.e.\
\[
z^t \;=\; z\cdot(z^{2c}a^rb^s)^{-1}
\;=\; z\cdot a^{-r}b^{-s}\cdot z^{-2c}.
\]
In the semidirect product \(B\rtimes\langle z\rangle\) the rewriting rule
\(\gamma\,z^j = z^j\,\gamma^{Z^{j}}\) holds for every \(\gamma\in B\) and
\(j\in\Z\); applied to \(\gamma = a^{-r}b^{-s}\) and \(j = -2c\), it yields
\(a^{-r}b^{-s}\cdot z^{-2c} = z^{-2c}\cdot(a^{-r}b^{-s})^{Z^{-2c}}\).
Hence
\begin{equation}\label{eq:zt-general}
	z^t \;=\; z\cdot z^{-2c}\cdot (a^{-r}b^{-s})^{Z^{-2c}}
	\;=\; z^{1-2c}\cdot (a^{-r}b^{-s})^{Z^{-2c}}.
\end{equation}
This identifies the \(B\)-component
\(\beta = (a^{-r}b^{-s})^{Z^{-2c}}\) of \(z^t\).   The
closed form~\eqref{eq:zt-general} is the basis for the identities
\((\mathcal I_3)\) and \((\mathcal I_4)\) below.

\paragraph{Identity \((\mathcal I_3)\): \((z^t)^t = z\) and the order of	\(\bar z\), giving~\eqref{eq:W4}.}
Applying the conjugation by \(t\) to~\eqref{eq:zt-general} modulo\(B\), and using \(t^2 = 1\), we obtain
\( \bar{z}=(\bar{z})^{t^2}= \bar{z}^{(1-2c)^2}\), so that
\[
(1-2c)^2\equiv 1\pmod M,
\]
which is~\eqref{eq:W4}.   

\paragraph{Identity \((\mathcal I_4)\): the base part of \((z^t)^t = z\),
	giving~\eqref{eq:W5}.}
We now retain the  \(B\)-part of the identity \(z = (z^t)^t\).   Recalling that \(k := 1-2c\) and
\(\beta := (a^{-r}b^{-s})^{Z^{-2c}}\) for the \(B\)-component
of~\eqref{eq:zt-general}, we can write \(z^t = z^k\beta\).   Applying \(t\) and
using \(t^2 = 1\):
\begin{equation}\label{eq:ztt}
z = (z^t)^t = (z^k\beta)^t = (z^t)^k\,\beta^t = (z^k\beta)^k\,\beta^t.
\end{equation}

We need to expand the power \(\bigl(z^{k}\beta\bigr)^{k}\) inside the semidirect
product \(B\rtimes\langle z\rangle\), in which \(\beta\,z = z\,\beta^{Z}\)
for \(\beta\in B\), so that \(z^{j}\beta = \beta^{Z^{-j}}z^{j}\).
Notice that, computing the square of \(z^k \beta\), we get
\[(z^k \beta) (z^k \beta) = (z^k \beta) (\beta^{Z^{-k}} z^k ) = z^k (\beta^{I+Z^{-k}}z^k)= z^k (z^k \beta^{(I+Z^{-k})Z^k})=z^{2k} \beta^{I+Z^k}.\]
In general 
collecting the \(k\) base factors past the powers of \(z\), we obtain 
\begin{equation}\label{eq:power-expand}
	\bigl(z^{k}\beta\bigr)^{k}
	= z^{k^2}\,\beta^{\,I + Z^{k} + Z^{2k} +
		\cdots + Z^{(k-1)k}}
	= z^{k^2}\,\beta^{S_k(Z^{k})},
\end{equation}
where \(S_k(X) := I + X + \cdots + X^{k-1}\) and we used that the
exponents of \(z\) accumulate as \(k+k+\cdots+k = k^2\) while each
successive base factor is conjugated by an additional \(Z^{k}\).
By~\eqref{eq:W4}, \(k^2\equiv 1\pmod M\), so the \(z\)-part
of~\eqref{eq:power-expand} is exactly \(z\).   Substituting
into Equation~\ref{eq:ztt} and canceling \(z\), we have the base
identity
\[
\beta^{S_k(Z^k)}\cdot\beta^t = 1\quad\text{in \(B\),}
\]
which in additive notation 
reads
\begin{equation}\label{eq:W5-from-I4}
	\bigl(S_k(Z^k)+J\bigr)\binom{r}{s}\equiv\binom{0}{0}\pmod N,
\end{equation}
that is~\eqref{eq:W5}.


\paragraph{Identity \((\mathcal I_5)\): the off-diagonal condition, via
	Cayley--Hamilton, giving~\eqref{eq:W3}}
Recall that~\eqref{eq:W3}  is the precise additional
content of the matrix identity \(Z^M = I\) beyond the scalar eigenvalue
conditions \eqref{eq:W1}--\eqref{eq:W2}.   

By Lemma~\ref{lem:IJ}, writing \(Z^M = A_M Z + B_M I\), with
sequences \(A_k, B_k\) determined by 
\[
A_{k+1} = (\operatorname{tr}Z)\,A_k + B_k,
\qquad
B_{k+1} = -(\det Z)\,A_k,
\]
 where
\(\operatorname{tr}Z = \sigma+\mu\) and \(\det Z = \sigma\mu\), the identity
\(Z^M = I\) is equivalent to the two scalar conditions
\begin{equation}\label{eq:offdiag}
	A_M\,q\equiv 0 \pmod N
	\quad\text{(off-diagonal)},
	\qquad
	A_M(1+p) + B_M\equiv 1 \pmod N
	\quad\text{(diagonal)}.
\end{equation}
Since \(\sigma,\mu\) are the roots of the characteristic polynomial
\(x^2 - (\operatorname{tr}Z)x + \det Z\), 
the sequence \(A_k=(\operatorname{tr}Z)\,A_{k-1} -(\det Z)\,A_{k-2}\) is the associated Lucas \(U\)-sequence; 
for \(\sigma\ne\mu\), it has the closed form
\begin{equation}\label{eq:lucas}
	A_M = \frac{\sigma^M - \mu^M}{\sigma - \mu}\,.
\end{equation}
Now \(\sigma - \mu = 2q\).   Multiplying~\eqref{eq:lucas} by \(q\) and using
that the fraction is an integer division in \(\Z\) (since \(A_M\in
\Z\) via the Lucas recurrence, and the right-hand side equals
\(A_M\cdot q\)), we obtain
\[
A_M\,q \;=\; \frac{\sigma^M-\mu^M}{\sigma-\mu}\cdot\frac{\sigma-\mu}{2}
\;=\; \frac{\sigma^M-\mu^M}{2} 
\]
note that the quantity \(\frac{\sigma^M-\mu^M}{2}\) is an integer because
\(\sigma-\mu\) is even, and so \(\sigma^M-\mu^M\) is even.   Reducing
modulo \(N\), the off-diagonal condition \(A_M q\equiv 0\pmod N\)
in~\eqref{eq:lucas} becomes
\[
\frac{\sigma^M-\mu^M}{2}\equiv 0\pmod N,
\quad\text{that is,}\quad
\sigma^M\equiv\mu^M\pmod{2N},
\]
which is~\eqref{eq:W3}.   When \(\sigma = \mu\), i.e.\ \(q = 0\), the matrix \(Z\) is diagonal,
\eqref{eq:W3} reduces to the tautology \(\sigma^M=\sigma^M\), and
\(Z^M = I\) is equivalent to \eqref{eq:W1} alone.

\paragraph{Identity \((\mathcal I_6)\): symmetric circulant constraint, giving~\eqref{eq:W6}.}
The presentation~\eqref{eq:main-pres} forces the \(z\)-action matrix to
take the symmetric circulant form \(Z = \left(\begin{smallmatrix}1+p&q
	\\q&1+p\end{smallmatrix}\right)\), which commutes with the swap \(J\).
Combined with Lemma~\ref{lem:circulant}, which gives \(JZJ = Z^k\) for
every exact product with \(B\) normal, this forces
\[
Z \;=\; J Z J \;=\; Z^{k},
\]
equivalently \(Z^{k-1} = I\) in \(\mathrm{GL}_2(\Z_N)\).   Since \(k - 1 =
-2c\) and \(Z\in\mathrm{GL}_2(\Z_N)\) is invertible, this is
\[
Z^{2c} \equiv I \pmod N,
\]
which is~\eqref{eq:W6}.   Equivalently, on the eigenline basis
\(\{u,v\}\) of \(J\) -- where \(Z\) acts as \(\sigma\) and \(\mu\) respectively
-- this amounts to the pair of scalar conditions \(\sigma^{2c}\equiv 1\)
and \(\mu^{2c}\equiv 1\pmod N\).

As a consequence of~\eqref{eq:W6}, the closed form~\eqref{eq:zt-general}
simplifies since \((a^{-r}b^{-s})^{Z^{-2c}} =
a^{-r}b^{-s}\)
 and therefore
\begin{equation}\label{eq:zt-simple}
	z^t \;=\; z^{1-2c}\,a^{-r}b^{-s}.
\end{equation}
This simplified form will be used freely below.

\paragraph{Identity \((\mathcal I_7)\): \((z^t)^M = 1\), giving~\eqref{eq:W7}.}
Since \(z^t\) is conjugate to \(z\) in \(G\) and \(z\) has order \(M\), the power
relation \((z^t)^M = 1\) must hold.   Using~\eqref{eq:zt-simple}, write
\(z^t = z^k\,\beta_0\) with \(k = 1-2c\) and \(\beta_0 := a^{-r}b^{-s}\), and
expand in \(B\rtimes\langle z\rangle\) using the same collection rule as
in \((\mathcal I_4)\):
\[
(z^k\beta_0)^M = z^{kM}\,\beta_0^{\,I + Z^k + Z^{2k} + \cdots + Z^{(M-1)k}}
= z^{kM}\,\beta_0^{S_M(Z^k)}.
\]
Since \(z^M = 1\), \(z^{kM} = 1\).   Hence \((z^t)^M = 1\) is equivalent to
\(\beta_0^{S_M(Z^k)} = 1\) in \(B\), which in additive notation
(\(\beta_0\) corresponds to \(-\binom{r}{s}\)) reads
\[
S_M(Z^k)\binom{r}{s} \equiv \binom{0}{0}\pmod N,
\]
which is~\eqref{eq:W7}.\qed

\subsection{Sufficient condition via consistency theorem}\label{sec:suff}

Before invoking the consistency theorem, we recall that the groups
under consideration are polycyclic and exhibit the series underlying the
PC-presentation.

\begin{lemma}\label{lem:polycyclic}
	The wreathed \(2\)-group \(W = \Wr_{2^n} = (C_{2^n}\times C_{2^n})\rtimes
	C_2\) is polycyclic.   More precisely, with \(B = \langle a\rangle\times
	\langle b\rangle\) its base and \(t\) the swap involution, the subnormal
	series
	\begin{equation}\label{eq:W-series}
		W \;\trianglerighteq\; B \;\trianglerighteq\; \langle a\rangle
		\;\trianglerighteq\; \{1\}
	\end{equation}
	has cyclic quotients \(W/B\cong C_2\), \(B/\langle a\rangle\cong C_{2^n}\)
	and \(\langle a\rangle\cong C_{2^n}\).   Consequently any exact product
	\(G = HK\) with \(H\cong\Wr_{2^n}\) and \(K = \langle z\rangle\cong C_{2^m}\)
	in which the base \(B\) is normal is polycyclic, with the following refined series
		\begin{equation}\label{eq:G-series}
		G \;\trianglerighteq\; \langle B, z\rangle
		\;\trianglerighteq\; B = \langle a,b\rangle
		\;\trianglerighteq\; \langle a\rangle
		\;\trianglerighteq\; \{1\}.
	\end{equation}
\end{lemma}

\begin{proof}
	Since \(B\) is abelian and \(W\) is a semidirect product \((C_{2^n}\times C_{2^n})\rtimes
	C_2\), we have that~\eqref{eq:W-series} is a subnormal series with syclic quotients.
	
	For the exact product \(G = HK\) with \(B\) normal, the base \(B\) is normal
	in \(G\) by hypothesis, and the quotient \(G/B\) has order \(|G|/|B| = 2M\).
	It is generated by the images \(\bar z\) (of order  \(M\)) and
	\(\bar t\) (of order \(2\)); the subgroup \(\langle\bar z\rangle\) has index
	\(2\) in \(G/B\), hence is normal, so \(G/B\) is an extension of
	\(\langle\bar z\rangle\cong C_{M}\) by \(\langle\bar t\rangle\cong C_2\),
	that is, abelian (when \(2c\equiv 0\pmod M\)) or non-abelian of order
	\(2M\) with \(\langle\bar z\rangle\) a cyclic subgroup of index \(2\), 
	by Proposition~\ref{prop:zt}.   Refining \(B\trianglelefteq G\) through the normal
	subgroup \(\langle\bar z\rangle\) of \(G/B\) and then through the
	series~\eqref{eq:W-series} of \(B\) gives the polycyclic series
	\begin{equation}
		G \;\trianglerighteq\; \langle B, z\rangle
		\;\trianglerighteq\; B = \langle a,b\rangle
		\;\trianglerighteq\; \langle a\rangle
		\;\trianglerighteq\; \{1\},
	\end{equation}
	with cyclic quotients \(C_2\) (namely \(G/\langle B,z\rangle =
	\langle\bar t\rangle\)), \(C_{2^m}\) (namely \(\langle B,z\rangle/B =
	\langle\bar z\rangle\)), \(C_{2^n}\) and \(C_{2^n}\) from top to bottom.
	Hence \(G\) is polycyclic of order \(2^{m}\cdot 2^{2n+1} = 2N^2M\).
	
	We remark that the polycyclicity of \(G\) also follows abstractly from
	\(G\) being a finite \(2\)-group -- hence nilpotent, so solvable, and a
	finite solvable group is polycyclic. 
	The explicit series~\eqref{eq:G-series}, however,
	is the one adapted to the generators of the PC-presentation, valid
	precisely because \(B\trianglelefteq G\), and it is this adaptation that
	we use below.
\end{proof}

The series~\eqref{eq:G-series} is the one underlying our
PC-presentation.   We now view~\eqref{eq:main-pres} as a polycyclic
presentation, also called PC-presentation
(see~\cite[Section~2.3]{Eick2000}
or~\cite[Definition~8.7]{HoltEickOBrien2005}), and verify that it
satisfies all the relevant consistency conditions (see~\cite[Lemma~2.5]{Eick2000}).

Order the generators according to the series~\eqref{eq:G-series}, with
\( t\) on top, then \( z\), then the base:
\[
g_1 := t,\quad g_2 := z,\quad g_3 := b,\quad g_4 := a,
\]
with relative orders \(e_1 = 2\), \(e_2 = M\), \(e_3 = N\), \(e_4 = N\).   Each
element of \(G\) has a unique normal form
\(g_1^{d_1}g_2^{d_2}g_3^{d_3}g_4^{d_4} = t^{d_1}z^{d_2}b^{d_3}a^{d_4}\)
with \(0\le d_i < e_i\).   With this ordering the conjugate of the
lower generator \(g_2 = z\) by the higher \(g_1 = t\) is expressed in normal
form using Proposition~\ref{prop:zt}, which is the reason the closed
form~\eqref{eq:zt-general} -- rather than the defining relation \(t^z\)
-- is the natural datum for the consistency check.

By condition~\eqref{eq:W6}, \(Z^{2c}\equiv I\pmod N\), so
\((a^{-r}b^{-s})^{Z^{-2c}} = a^{-r}b^{-s}\) and~\eqref{eq:zt-general}
simplifies to the closed form~\eqref{eq:zt-simple}, namely
\(z^t = z^{1-2c}\,a^{-r}b^{-s}\).   The conjugate relations (for \(i<j\),
with the result in normal form involving generators of index \(\ge i\))
are
\begin{equation}\label{eq:conj-relations}
	\begin{aligned}
		g_2^{g_1} &= z^t = z^{1-2c}a^{-r}b^{-s}
		= g_2^{1-2c}g_3^{-s}g_4^{-r},\\
		g_3^{g_1} &= b^t = a = g_4,\\
		g_4^{g_1} &= a^t = b = g_3,\\
		g_3^{g_2} &= b^z = a^{q}b^{1+p} = g_3^{1+p}g_4^{q},\\
		g_4^{g_2} &= a^z = a^{1+p}b^{q} = g_3^{q}g_4^{1+p},\\
		g_4^{g_3} &= a^b = a = g_4 \quad (\text{since } [a,b]=1).
	\end{aligned}
\end{equation}
The power relations are \(g_i^{e_i} = 1\) for all \(i\).

By the consistency theorem for polycyclic presentations
(see~\cite[Lemma~2.5]{Eick2000}, with the proof attributed to
\cite[p.~424]{Sims}), the PC-presentation is consistent -- equivalently,
defines a group of order \(e_1e_2e_3e_4 = 2N^2M\) -- if and only if a
finite list of overlap identities, derived from the rewriting rules,
reduce to the same normal form along all reduction paths.   We organise
them, as in~\cite[Lemma~2.5(c)]{Eick2000}, considering the following consistency identities
\begin{itemize}
	\item[(P)] Identities \(g_j^{\,g_i^{e_i}} = g_j\), for \(j >i\)
	\item[(W)] Identities \(\bigl(g_j^{g_i}\bigr)^{e_j} = \bigl(g_j^{e_j}\bigr)^{g_i}\),  for \(j >i\)
	\item[(C)] Identities \(g_k^{\,g_jg_i} = (g_k^{\,g_j})^{g_i}\), for \(k>j>i\)
\end{itemize}
and verify that under \eqref{eq:W1}--\eqref{eq:W7} (with \(M\ne 4\))
each holds.

\begin{remark}\label{rem:corcon}
	The families (P) and (W) are respectively derived from the third and the second identities in \cite[Lemma~2.5(c)]{Eick2000}, namely for \(i < j\)
	\[
	g_j(g_i^{x_i})=(g_jg_i)g_i^{x_{i-1}},
	\qquad
	(g_j^{x_j})g_i=g_j^{x_{j-1}}(g_jg_i).
	\]
	Note that in our setting, the exponents \(x_i\) coincide with the orders of the corresponding generators. Moreover, repeatedly applying the identity
	\(
	g_jg_i=g_i\,g_j^{g_i}
	\)  yields
	\[
	(g_jg_i)g_i^{x_{i-1}}
	=(g_ig_j^{g_i})g_i^{x_{i-1}}
	=(g_i^2g_j^{g_i^2})g_i^{x_{i-2}}
	=\cdots
	=g_i^{x_i}g_j^{g_i^{x_i}},
	\]
	from which the family (P) follows; the derivation of (W) is analogous.
	
	The family (C) follows immediately from the first identity \(g_k(g_j g_i)=(g_kg_j)g_i\), with \(k>j>i\), in \cite[Lemma~2.5]{Eick2000}. As for the remaining two identities in that lemma, the fourth \((g_i^{x_i})g_i=g_i(g_i^{x_i})\) is trivially satisfied in our setting, while the last does not apply because all orders \(e_i\) are finite.
\end{remark}

\paragraph{(P) Identities \(g_j^{\,g_i^{e_i}} = g_j\), for \(i<j\).}
Iterating the conjugation rule \(g_j^{g_i}\) a total of \(e_i\) times must
return \(g_j\).

\emph{Pairs with \(i = 1\) (conjugation by \(t\), twice).}   We must check
\(g_j^{\,t^2} = g_j\) for \(j = 2,3,4\).   For \(j = 3,4\) this is \(J^2 = I\) on
the base, which holds identically.   For \(j = 2\) it is \((z^t)^t = z\).
Using~\eqref{eq:conj-relations}, \((z^t)^t = (z^{1-2c}a^{-r}b^{-s})^t
= (z^t)^{1-2c}a^{-s}b^{-r}\); expanding the power as in \((\mathcal I_4)\)
and using \eqref{eq:W4} (\(k^2\equiv 1\bmod M\), so the \(z\)-part is \(z\))
and \eqref{eq:W5} (so the accumulated base part is trivial) gives
\((z^t)^t = z\), as required.   The simplified closed
form~\eqref{eq:zt-simple} used here is itself valid by~\eqref{eq:W6}.
This is exactly where \eqref{eq:W4} and~\eqref{eq:W5} are used.

\emph{Pairs with \(i = 2\) (conjugation by \(z\), \(M\) times).}   We must
check \(g_j^{\,z^M} = g_j\) for \(j = 3,4\), which is the matrix identity
\(Z^M = I\) on the base.   By \((\mathcal I_1)\) and \((\mathcal I_5)\) this is
equivalent to \eqref{eq:W1}, \eqref{eq:W2} and \eqref{eq:W3} together.

\emph{Pair \((3,4)\).}   \(a^{\,b^N} = a\) holds identically since
\([a,b]=1\).

\begin{center}
	\begin{tabular}{lll}
		\toprule
		Pair \((i,j)\) & Reduces to & Outcome \\
		\midrule
		\((1,2)\) & \((z^t)^t = z\) & \eqref{eq:W4} \mbox{ and } \eqref{eq:W5} \\
		\((1,3)\) & \(b^{t^2} = b\) & holds, \(J^2 = I\) \\
		\((1,4)\) & \(a^{t^2} = a\) & holds, \(J^2 = I\) \\
		\((2,3)\) & \(b^{z^M} = b\) & part of \(Z^M = I\): \eqref{eq:W1}--\eqref{eq:W3} \\
		\((2,4)\) & \(a^{z^M} = a\) & part of \(Z^M = I\): \eqref{eq:W1}--\eqref{eq:W3} \\
		\((3,4)\) & \(a^{b^N} = a\) & holds identically \\
		\bottomrule
	\end{tabular}
\end{center}

\paragraph{(W) Identities \(\bigl(g_j^{g_i}\bigr)^{e_j} = \bigl(g_j^{e_j}\bigr)^{g_i}\)
	reducing to power consistency, for \(i<j\).}
Raising the conjugate \(g_j^{g_i}\) to the power \(e_j\) must agree with the
power relation \(g_j^{e_j} = 1\) transported by \(g_i\).   For \((i,j)=(1,2)\)
this is \((z^t)^M = 1\). By~\eqref{eq:zt-simple} (valid by~\eqref{eq:W6}),
\(z^t = z^{1-2c}a^{-r}b^{-s}\), and the expansion in \(B\rtimes\langle
z\rangle\) used in \((\mathcal I_7)\) shows that
\((z^t)^M = z^{(1-2c)M}\,(a^{-r}b^{-s})^{S_M(Z^k)}\).   Since \(z^M = 1\)
the \(z\)-factor is trivial; the base factor vanishes precisely
by~\eqref{eq:W7}.   This is where~\eqref{eq:W7} is used.   The
remaining five pairs have their conjugate \(g_j^{g_i}\) lying in the
abelian base \(B\) (for \(i\ge 2\)) or equal to a single base generator
(for \(i=1\), \(j\ge 3\)), so the power relations \(a^N = b^N = 1\) apply
directly.

\begin{center}
	\begin{tabular}{lll}
		\toprule
		Pair \((i,j)\) & Reduces to & Outcome \\
		\midrule
		\((1,2)\) & \((z^t)^M = 1\) &  \eqref{eq:W7} (via~\eqref{eq:W6}) \\
		\((1,3)\) & \((b^t)^N = a^N = 1\) & from \(a^N=1\) \\
		\((1,4)\) & \((a^t)^N = b^N = 1\) & from \(b^N=1\) \\
		\((2,3)\) & \((b^z)^N = 1\) & from \(a^N=b^N=1\) \\
		\((2,4)\) & \((a^z)^N = 1\) & from \(a^N=b^N=1\) \\
		\((3,4)\) & \((a^b)^N = a^N = 1\) & from \(a^N=1\) \\
		\bottomrule
	\end{tabular}
\end{center}

\paragraph{(C) Identities \(g_k^{\,g_jg_i} = (g_k^{\,g_j})^{g_i}\),
	for \(i<j<k\).}
The two sides are the two reductions of the overlap \(g_kg_jg_i\), and
their equality expresses associativity of conjugation.   There are four
such triples.

\emph{Triples \((1,2,3)\) and \((1,2,4)\)} involve conjugation of a base
generator by both \(t\) and \(z\).   For \((1,2,4)\), we show that \(a^{\,zt} = (a^{z})^t\) by comparing the two reduction paths. Along the first path, conjugating by \(z\) first, we have,
\((a^z)^t = (a^{1+p}b^q)^t = b^{1+p}a^q\). In additive notation, identifying \(a\) with
\(\binom10\), this corresponds to \(\binom{1}{0}\stackrel{z}{\rightarrow}Z\binom{1}{0}\stackrel{t}{\rightarrow}JZ\binom{1}{0}=ZJ\binom{1}{0}=Z\binom{0}{1}\), where we have used that \(JZ=ZJ\). Along the second path, using \(a^{zt}= a^{tz^t}\), together with
\(z^t = z^k\gamma\) and the fact that \(B\) acts trivially on itself by conjugation, we obtain, \(a^{zt} = (a^t)^{z^t} = b^{z^k}\), additively
\(Z^k\binom{0}{1}\).   The two paths agree iff \(Z^k\binom{0}{1}
= Z\binom{0}{1}\); combined with the symmetric calculation for
\((1,2,3)\) which gives \(Z^k\binom{1}{0} = Z\binom{1}{0}\), this requires
\(Z^{k-1} = I\) as a matrix in \(\mathrm{GL}_2(\Z_N)\).  This requirement is given by~\eqref{eq:W6}, so both triples hold without any further constraint.

\emph{Triples \((1,3,4)\) and \((2,3,4)\)} involve the commuting base
generators.   Indeed, \((a^b)^t = a^t = b = (a^t)^b\) and \((a^b)^z = a^z = (a^z)^b\),
both from \([a,b]=1\).

\begin{center}
	\begin{tabular}{lll}
		\toprule
		Triple \((i,j,k)\) & Reduces to & Outcome \\
		\midrule
		\((1,2,3)\) & \(b^{\,zt} = (b^z)^t\) & from \eqref{eq:W6} \\
		\((1,2,4)\) & \(a^{\,zt} = (a^z)^t\) & from \eqref{eq:W6} \\
		\((1,3,4)\) & \((a^b)^t = (a^t)^b\) & from \([a,b]=1\) \\
		\((2,3,4)\) & \((a^b)^z = (a^z)^b\) & from \([a,b]=1\) \\
		\bottomrule
	\end{tabular}
\end{center}

\subsection{Completing the proof}

By the consistency theorem for polycyclic presentations
(\cite[Lemma~2.5]{Eick2000} and~\cite[p.~424]{Sims}), the
PC-presentation~\eqref{eq:main-pres} satisfies all consistency
conditions, hence defines a group \(G\) of order exactly
\(e_1e_2e_3e_4 = 2N^2M\).

The subgroup \(H = \langle a,b,t\rangle\) consists of elements whose
normal form has exponent \(d_2 = 0\) (no \(z\)), so \(|H| = 2N^2\); combined with the
imposed relations \(a^N = b^N = t^2 = 1\), \([a,b]=1\), \(a^t = b\), \(b^t = a\),
this gives \(H\cong\Wr_{2^n}\).   The subgroup \(K = \langle z\rangle\) has
order \(M\), so \(K\cong C_{2^m}\).   Their intersection \(H\cap K\) has both
\(d_2 = 0\) (from \(H\)) and \(d_1 = d_3 = d_4 = 0\) (from \(K\)), so it is
trivial.   Therefore \(|HK| = |H|\cdot|K|/|H\cap K| = 2N^2M = |G|\), and
\(G = HK\) is an exact product.   The conjugation rules
\(a^z, b^z\in B\) show that \(B^z = B\), and \(B^t = B\), so \(B\) is normal in
\(G\).   Finally, by Proposition~\ref{prop:zt}, \(H\) is normal in \(G\) when
\(2c\equiv 0\pmod M\) (and then \(G = H\rtimes K\) is split) and not normal
otherwise (and then \(G\) is a non-split exact product, with \(G/B\)
non-abelian of order \(2M\) with a cyclic subgroup of index \(2\)). \qed

\medskip



\section{Recovering the split case when \(M\ne 4\)}\label{sec:split-recover}

In this section we show that, in the split case, the general system of congruences \eqref{eq:W1}--\eqref{eq:W7} given in Theorem~\ref{thm:main} reduces to only five conditions. Throughout this section, we assume that \(M \not= 4\), as this exceptional case will be treated separately later.

First we need the following technical result
\begin{lemma}\label{lem:SM-valuation}
	For every odd integer \(\mu\) and every \(M = 2^m\) with \(m\ge 1\), the
	geometric sum \(S_M(\mu) = 1 + \mu + \mu^2 + \cdots + \mu^{M-1}\) is
	divisible by \(\frac{M(\mu+1)}{2}\) in \(\Z\), with \(2\)-adic valuation
	\[
	v_2\bigl(S_M(\mu)\bigr) \;=\; v_2\bigl(\mu+1\bigr) + (m-1)
	\;=\; v_2\!\left(\tfrac{M(\mu+1)}{2}\right).
	\]
	Equivalently, \(S_M(\mu) / \tfrac{M(\mu+1)}{2}\) is an odd integer.
\end{lemma}

\begin{proof}
	From the identity \((\mu-1)\,S_M(\mu) = \mu^M - 1\) in \(\Z\), we
	have \(v_2(S_M(\mu)) = v_2(\mu^M-1) - v_2(\mu-1)\).   It remains to
	compute \(v_2(\mu^M-1)\).
	
	Since \(M = 2^m\), we factor
	\[
	\mu^M - 1 \;=\; (\mu-1)(\mu+1)\,\prod_{i=1}^{m-1}\bigl(\mu^{2^i}+1\bigr).
	\]
	The first two factors contribute \(v_2(\mu-1) + v_2(\mu+1)\).   For
	\(i\ge 1\), since \(\mu\) is odd we have 
	\(\mu^{2^i}+1\equiv 2\pmod 4\)),   hence
	\(v_2(\mu^{2^i}+1) = 1\).   Summing over \(i = 1,\dots,m-1\) contributes
	\(m-1\) to the valuation.   Therefore
	\[
	v_2(\mu^M-1) \;=\; v_2(\mu-1) + v_2(\mu+1) + (m-1),
	\]
	and subtracting \(v_2(\mu-1)\):
	\[
	v_2(S_M(\mu)) \;=\; v_2(\mu+1) + (m-1).
	\]
	On the other hand, \(v_2\!\left(\tfrac{M(\mu+1)}{2}\right) = m +
	v_2(\mu+1) - 1 = v_2(\mu+1) + (m-1)\).   Hence
	\(\tfrac{M(\mu+1)}{2}\) divides \(S_M(\mu)\) in \(\Z\) and the quotient is
	odd.
\end{proof}

Now, we use this lemma for showing that the general system
\eqref{eq:W1}--\eqref{eq:W7} reduces, in this split sub-case, to the six
congruences classifying the split products, namely
\begin{align}
	\sigma^M &\equiv 1\pmod N, &
	\mu^M &\equiv 1\pmod N, &
	\sigma^M &\equiv\mu^M\pmod{2N}, \tag{W1--W3}\\
	r+s &\equiv 0\pmod N, &
   (r-s)M(\mu+1) &\equiv 0\pmod{4N}. \tag{W4\('\) , W5\('\)}
\end{align}

When \(2c\equiv 0\pmod M\) (equivalently \(k = 1-2c\equiv 1\pmod M\)),
Proposition~\ref{prop:zt} gives \(z^t\equiv z\pmod B\),
 \(G/B\) is abelian and \(H\) is normal, hence the product is split,
\(G = H\rtimes K\).   For \(M\le 4\) this coincides with ``\(c\) even''; for
\(M\ge 8\) it is strictly stronger, since \(c\) even also admits
\(k = M/2+1\ne 1\), for which \(G/B\) is non-abelian.

Conditions (W1)--(W3) are common to both formulations.   Since \(k\equiv 1\pmod M\), also \(k^2\equiv 1\pmod M\), and so \eqref{eq:W4} imposes no constraint; also \eqref{eq:W6} (\(Z^{2c} = Z^0 = I\)) is trivial.   The non-trivial reduction is thus from~\eqref{eq:W5}
and~\eqref{eq:W7} to (W4\('\)) and (W5\('\)).

Choosing the canonical residue \(k = 1\) in \(\{0,1,\dots,M-1\}\), the
sum \(S_k(Z^k) = S_1(Z) = I\).
By~\eqref{eq:zt-simple}, for
\(2c \equiv 0 \pmod M\), \(z^t = z\,a^{-r}b^{-s}\).   Hence, since \(a^t = b\)
and \(b^t = a\), we obtain
\[
(z^t)^t \;=\; (z\,a^{-r}b^{-s})^t
\;=\; z^t\,(a^{-r}b^{-s})^t
\;=\; z\,a^{-r}b^{-s}\cdot a^{-s}b^{-r}
\;=\; z\,a^{-(r+s)}\,b^{-(r+s)},
\]
and imposing \((z^t)^t = z\) forces \(a^{-(r+s)}b^{-(r+s)} = 1\) in \(B\),
i.e.,
\begin{equation}\label{eq:W4prime-derived}
	r + s \equiv 0 \pmod N,
\end{equation}
which is (W4\('\)).   The same conclusion is recovered by
applying~\eqref{eq:W5} with \(k \equiv 1 \pmod M\), since the operator
\(S_k(Z^k) + J\) becomes \(I + J\), and
\((I+J)\tbinom rs = \tbinom{r+s}{r+s} \equiv 0\pmod N\) gives back
\eqref{eq:W4prime-derived}.
In particular we highlight that \(r+s\equiv 0\pmod N\) with \(N\) even forces \(r\) and \(s\) to have the same parity, so that \(r-s\equiv 0\pmod 2\).
The remaining condition (W5\('\)) emerges
from~\eqref{eq:W7}.   In the split
sub-case, \(z^t = z\,a^{-r}b^{-s}\) and~\eqref{eq:W7} reads
\begin{equation}\label{eq:W7-split}
	S_M(Z)\binom{r}{s}\equiv\binom00\pmod N.
\end{equation}
Since \(Z\) is a symmetric circulant, its eigenvectors over \(\Z\) are
\(u = \binom{1}{1}\) and \(v = \binom{1}{-1}\), with \(Zu = \sigma u\) and
\(Zv = \mu v\); consequently \(S_M(Z)\,v = S_M(\mu)\,v\).

By~\eqref{eq:W4prime-derived}, we have
\(s\equiv -r\pmod N\), which gives
\[
\binom rs \;\equiv\; r\binom{1}{-1}\pmod N,
\]
and therefore, applying the integer matrix \(S_M(Z)\) and reducing
modulo \(N\),
\[
S_M(Z)\binom rs \;\equiv\; S_M(Z)\,r\binom{1}{-1}
\;=\; S_M(\mu)\,r\binom{1}{-1}\pmod N.
\]
Consequently, by~\ref{eq:W7-split} we obtain
\begin{equation}\label{eq:W7-antisym}
	S_M(\mu)\,r \equiv 0 \pmod N.
\end{equation}


By Lemma~\ref{lem:SM-valuation}, \(S_M(\mu) = \tfrac{M(\mu+1)}{2}\,w\) with
\(w\in\Z\) odd; since \(N\) is a power of \(2\), \(w\) is a unit in \(\Z_N\), and
\eqref{eq:W7-antisym} is equivalent to
\(\tfrac{M(\mu+1)}{2}\,r\equiv 0\pmod N\), that is,
\begin{equation}\label{eq:W6-in-r}
	M(\mu+1)\,r \equiv 0 \pmod{2N}.
\end{equation}

It remains to rewrite~\eqref{eq:W6-in-r} in the symmetric form
\textnormal{(W5\('\))}.   Take the canonical representatives
\(r,s\in\{0,\dots,N-1\}\).   Then \(r+s\equiv 0\pmod N\) leaves only the two
possibilities \(r = s = 0\) and \(r+s = N\). In the first case, \textnormal{(W6\('\))} is obvious, in the second case, we can write \(r-s = 2r-N\), and so
\[
M(\mu+1)(r-s) \;=\; 2\,M(\mu+1)\,r \;-\; N\,M(\mu+1).
\]
Since \(M\) is even and \(\mu\) is odd, \(M(\mu+1)\equiv 0\pmod 4\), whence
\(N\,M(\mu+1)\equiv 0\pmod{4N}\) and
\[
M(\mu+1)(r-s) \;\equiv\; 2\,M(\mu+1)\,r \pmod{4N}.
\]
Comparing
with~\eqref{eq:W6-in-r}, we obtain the equivalence
\[
M(\mu+1)\,r \equiv 0 \pmod{2N}
\quad\Longleftrightarrow\quad
M(\mu+1)(r-s) \equiv 0 \pmod{4N},
\]
whose right-hand side is \textnormal{(W5\('\))}.   

 Thus
the system~\eqref{eq:W5},~\eqref{eq:W7} for \(2c\equiv 0\pmod M\) is
exactly the conjunction \textnormal{(W4\('\)), (W5\('\))}, and
Theorem~\ref{thm:main} contains the split classification as its split
specialisation.

\section{The non-symmetric case \(JZ\ne ZJ\)}\label{sec:nonsym}

We now 
deal with the exact products \(G = HK\) with \(B\)
normal in which the matrix \(Z\) of the \(z\)-action on \(B\)
satisfies \(JZJ = Z^{k}\) (Equation~\eqref{eq:JZJ-epsilon}) but does not commute with the swap
\(J\).   These products are necessarily non-split, indeed by
Lemma~\ref{lem:circulant} \(JZJ = Z^{k}\), so the split case \(k = 1\)
(equivalently \(2c\equiv 0\pmod M\)) forces \(JZJ = Z\), i.e.\ \(JZ = ZJ\);
contrapositively, \(JZ\ne ZJ\) implies \(k\ne 1\) and the product is
non-split.   
We give a uniform
treatment of the entire non-symmetric branch in entry form, free of
any a priori symmetry assumption on \(Z\).

\subsection*{Notation and Cayley--Hamilton sequences for a general \(Z\)}

Throughout this section we drop the parametrisation
\(Z = \begin{psmallmatrix} 1+p & q \\ q & 1+p \end{psmallmatrix}\)
of~\eqref{eq:Zmatrix} and write \(Z\) directly in entry form
\begin{equation}\label{eq:Z-entries}
	Z = \begin{pmatrix} \alpha & \beta \\ \gamma & \delta \end{pmatrix}
	\in \mathrm{GL}_2(\Z_N).
\end{equation}
The PC-presentation of \(G\) becomes
\begin{equation}\label{eq:pres-entries}
	G = \left\langle a, b, t, z \,\middle|\,
	\begin{array}{l}
		a^N = b^N = t^2 = z^M = 1,\ [a,b] = 1,\\
		a^t = b,\quad b^t = a,\\
		a^z = a^{\alpha} b^{\beta},\quad b^z = a^{\gamma} b^{\delta},\\
		t^z = t\,z^{2c}\,a^r\,b^s
	\end{array}
	\right\rangle
\end{equation}
with \(\alpha,\beta,\gamma,\delta,r,s\in\Z_N\) and \(c\in\Z_M\).   

Our convention throughout this section is that \(Z\) acts on row exponent
vectors by right multiplication, that is,
\((a^ib^j)^z = (i,j)Z = (i\alpha+j\gamma,\; i\beta+j\delta)\).

\medskip

The trace and determinant of \(Z\) are
\[
\tau := \alpha + \delta= \operatorname{tr}Z,
\qquad
\nu := \alpha\delta - \beta\gamma \;=\; \det Z \in (\Z_N)^*.
\]
By the Cayley--Hamilton theorem \(Z^2 = \tau Z - \nu I\), so \(Z^j\) lies
in the \(\Z_N\)-span of \(Z\) and \(I\) for every \(j\ge 0\).   Define
integer sequences \((A_j)_{j\ge 0}, (B_j)_{j\ge 0}\) by
\begin{equation}\label{eq:AB-rec}
	A_0 = 0,\quad A_1 = 1,\quad A_{j+1} = \tau\,A_j - \nu\,A_{j-1},
	\qquad
	B_0 = 1,\quad B_j = -\nu\,A_{j-1}\ (j\ge 1).
\end{equation}
Then \(Z^j = A_j\,Z + B_j\,I\) for all \(j\ge 0\), and the four entries of
\(Z^j\) are
\begin{equation}\label{eq:Zj-entries}
	(Z^j)_{11} = A_j\alpha + B_j,\quad
	(Z^j)_{12} = A_j\beta,\quad
	(Z^j)_{21} = A_j\gamma,\quad
	(Z^j)_{22} = A_j\delta + B_j.
\end{equation}


\subsection*{The main theorem}

The following theorem is stated in entry form and covers both
the symmetric and non-symmetric cases.   It generalises
Theorem~\ref{thm:main} and is the main result of this section.

\begin{theorem}[Unified classification, \(M\ne 4\)]\label{thm:nonsym}
	Let \(n\ge 1\), \(m\ge 1\), \(N = 2^n\), \(M = 2^m\ne 4\).   A tuple
	\((\alpha,\beta,\gamma,\delta,r,s,c)\in\Z_N^4\times\Z_N^2\times\Z_M\)
	defines via the PC-presentation~\eqref{eq:pres-entries} an exact
	product \(G = HK\) with \(H\cong\Wr_{2^n}\), \(K\cong C_{2^m}\) and \(B\)
	normal in \(G\) if and only if, setting \(k := 1-2c\), the following conditions
 	are satisfied
	\begin{align}
		& A_M\,\alpha + B_M \equiv 1,\quad
		A_M\,\beta \equiv 0,\quad
		A_M\,\gamma \equiv 0,\quad
		A_M\,\delta + B_M \equiv 1 \pmod N,
		\tag{N1}\label{eq:N1}\\
		& A_k\,\alpha + B_k \equiv \delta,\quad
		A_k\,\beta \equiv \gamma,\quad
		A_k\,\gamma \equiv \beta,\quad
		A_k\,\delta + B_k \equiv \alpha \pmod N,
		\tag{N2}\label{eq:N2}\\
		& (1-2c)^2 \equiv 1 \pmod M,
		\tag{N3}\label{eq:N3}\\
		& \bigl(r,\,s\bigr) Z^{-2c}\bigl(S_k(Z^k) + J\bigr)
		\equiv (0,0) \pmod N,
		\tag{N4}\label{eq:N4}\\
		& \bigl(r,\,s\bigr) Z^{-2c} S_M(Z^k)
		\equiv (0,0) \pmod N,
		\tag{N5}\label{eq:N5}
	\end{align}
	where \(S_l(X) := I + X + \cdots + X^{l-1}\), for any positive integer \(l\).
	Equivalently, \eqref{eq:N4} reads
	\(\bigl(S_k(Z^k)^{\mathsf T} + J\bigr)(Z^{-2c})^{\mathsf T}\binom{r}{s}
	\equiv \binom00\pmod N\).
\end{theorem}

Notice that, when \(Z\) is symmetric circulant, \(S_k(Z^k)^{\mathsf T} =
S_k(Z^k)\) and Theorem~\ref{thm:nonsym} specialises to
Theorem~\ref{thm:main}. Suppose \(\alpha = \delta\) and \(\beta = \gamma\), so
\(Z = \begin{psmallmatrix} 1+p & q \\ q & 1+p \end{psmallmatrix}\) with
\(1+p = \alpha\), \(q = \beta\); then \(\tau = 2(1+p)\) and
\(\nu = (1+p)^2 - q^2 = \sigma\mu\), where \(\sigma = 1+p+q\) and
\(\mu = 1+p-q\) are the eigenvalues of \(Z\).   Conditions
\eqref{eq:N1}--\eqref{eq:N5} of Theorem~\ref{thm:nonsym} specialise to
the full system \eqref{eq:W1}--\eqref{eq:W7} of
Theorem~\ref{thm:main} through the following dictionary.
\begin{itemize}
	\item \eqref{eq:N1}, encoding \(Z^M = I\) entrywise, evaluated on the
	eigenvectors \(\binom11\) and \(\binom1{-1}\), gives the diagonal congruences
	\(\sigma^M\equiv 1\) and \(\mu^M\equiv 1\pmod N\) (these are
	\eqref{eq:W1}--\eqref{eq:W2}) together with the off-diagonal
	congruence \(A_M\,q\equiv 0\pmod N\); since
	\(\sigma^M-\mu^M = (\sigma-\mu)A_M = 2q\,A_M\), the latter is
	\(\sigma^M\equiv\mu^M\pmod{2N}\), i.e.\ \eqref{eq:W3}.
	\item \eqref{eq:N2}, since
	\(JZJ = Z\) for symmetric circulant \(Z\), encoding \(JZJ = Z^k\) entrywise becomes the relation \(Z = Z^k\), i.e.\ \(Z^{k-1} = Z^{-2c} = I\); this is \eqref{eq:W6}, \(Z^{2c}\equiv I\pmod
	N\).   
	\item \eqref{eq:N3} is just \eqref{eq:W4},
	\((1-2c)^2\equiv 1\pmod M\).
	\item \eqref{eq:N4} reduces to \eqref{eq:W5}. In the symmetric case
	\(S_k(Z^k)\) is symmetric (as \(Z = Z^{\mathsf T}\)) and \(Z^{-2c}\)
	commutes with \(S_k(Z^k)+J\), so multiplying by
	\(Z^{2c}\) turns the row form
	\((r,s) Z^{-2c}(S_k(Z^k)+J)\equiv 0 \pmod N\) into the column form
	\((S_k(Z^k)+J)\binom rs\equiv 0 \pmod N\).
	\item \eqref{eq:N5} reduces to \eqref{eq:W7} exactly as \eqref{eq:N4}
	reduced to \eqref{eq:W5}. Multiplying the row form
	\((r,s)\,Z^{-2c}S_M(Z^k)\equiv 0\) by  \(Z^{2c}\) and using
	that \(S_M(Z^k)\) is symmetric gives \(S_M(Z^k)\binom rs\equiv 0 \pmod N\), i.e.\ \eqref{eq:W7}.
\end{itemize}


\subsection*{Proof of necessary condition}

We use the same identities \((\mathcal{I}_1)\)--\((\mathcal{I}_5)\) that
drove the proof of Theorem~\ref{thm:main} in Section~\ref{sec:main-thm},
with \(Z\) no longer assumed symmetric circulant.

\paragraph{Identity \((\mathcal{I}_1')\): \(Z^M = I\) gives \eqref{eq:N1}.}
The relation \(z^M = 1\) forces \(Z^M = I\) in \(\mathrm{GL}_2(\Z_N)\).
Equating entries via~\eqref{eq:Zj-entries} at \(j = M\), we obtain
\(A_M\alpha + B_M = 1\), \(A_M\beta = 0\), \(A_M\gamma = 0\),
\(A_M\delta + B_M = 1\) in \(\Z_N\), which is~\eqref{eq:N1}.

\paragraph{Identity \((\mathcal{I}_2')\): the closed form for \(z^t\).}
The derivation of \((\mathcal I_2)\) in Section~\ref{sec:nec} never used
\(JZ = ZJ\); thus we obtain the same closed
form~\eqref{eq:zt-general}, namely
\(z^t = z^{1-2c}(a^{-r}b^{-s})^{Z^{-2c}}\).

\paragraph{Identity \((\mathcal{I}_3')\): \eqref{eq:N3} from \((z^t)^t = z\) mod \(B\).}
As done in \((\mathcal{I}_3)\), reducing \((z^t)^t = z\) modulo \(B\) gives
\((1-2c)^2 \equiv 1\pmod M\), i.e., \eqref{eq:N3}.

\paragraph{Identity \((\mathcal{I}_4')\): \eqref{eq:N4} from the base part of \((z^t)^t = z\).}
Write \(\beta_0 = (a^{-r}b^{-s})^{Z^{-2c}}\), so \(z^t = z^k\beta_0\) with
\(k=1-2c\) by~\eqref{eq:zt-general}.   In row-vector notation
\(\beta_0\) corresponds to  \((-r,-s) Z^{-2c}\).
We compute \((z^t)^t = (z^k\beta_0)^t = (z^t)^k\,\beta_0^t\)
and then expand the power on the right.   Within the semidirect
product structure of \(B\rtimes\langle z\rangle\) the rule is
\(\beta z = z\beta^Z\) for \(\beta\in B\), equivalently
\(z^j\beta = \beta^{Z^{-j}} z^j\).   Collecting the \(k\)
base factors past the \(z\)-powers in \((z^k\beta_0)^k\), we have
\begin{equation}\label{eq:collect}
	(z^k\beta_0)^k
	\;=\; z^{k^2}\,\beta_0^{I + Z^k + Z^{2k} + \cdots
	+ Z^{(k-1)k}}
	\;=\; z^{k^2}\,\beta_0^{S_k(Z^k)}.
\end{equation}
Using \(k^2\equiv 1\pmod M\) and \(z^M = 1\) we have \(z^{k^2} = z\).
Hence
\[
(z^t)^t \;=\; z\beta_0^{S_k(Z^k)}\beta_0^t
\;=\; z\beta_0^{S_k(Z^k) + J}.
\]
Setting \((z^t)^t = z\) forces the base correction to vanish, that is,
\(\beta_0^{S_k(Z^k) + J} = 1\) in \(B\).   In other words,
\[
(-r,-s) Z^{-2c}\bigl(S_k(Z^k) + J\bigr) \equiv (0,0) \pmod N,
\]
which is~\eqref{eq:N4}.

\paragraph{Identity \((\mathcal{I}_5')\): \eqref{eq:N2} from \(JZJ = Z^k\).}
Lemma~\ref{lem:circulant} gives \(JZJ = Z^k\) for \(k=1-2c\).   Writing
\(JZJ = \begin{psmallmatrix} \delta & \gamma \\ \beta & \alpha \end{psmallmatrix}\)
and \(Z^k = A_kZ + B_kI = \begin{psmallmatrix} A_k\alpha+B_k & A_k\beta \\
	A_k\gamma & A_k\delta+B_k \end{psmallmatrix}\), the four entrywise
equalities are precisely~\eqref{eq:N2}.

\paragraph{Identity \((\mathcal{I}_6')\): \eqref{eq:N5} from \((z^t)^M = 1\).}
Since \(z^t\) is conjugate to \(z\), which has order \(M\), the power
relation \((z^t)^M = 1\) holds.   With \(z^t = z^k\beta_0\) and
\(\beta_0 = (a^{-r}b^{-s})^{Z^{-2c}}\) as in \((\mathcal{I}_4')\), with the same reasoning in~\eqref{eq:collect} we obtain
\[
(z^k\beta_0)^M = z^{kM}\,\beta_0^ {S_M(Z^k)}.
\]
Since \(z^M = 1\) we have \(z^{kM} = 1\), so \((z^t)^M = 1\) forces
\(\beta_0^{S_M(Z^k)} = 1\) in \(B\), that is,

\[
(r,s) Z^{-2c} S_M(Z^k) \equiv (0,0)\pmod N,
\]
which is~\eqref{eq:N5}.

\subsection*{Proof of sufficient condition}

We verify the consistency
conditions~(P),~(W),~(C) of~\cite[Lemma~2.5(c)]{Eick2000} for the
PC-presentation~\eqref{eq:pres-entries} relative to the polycyclic
series~\eqref{eq:G-series} (for the corresponce between conditions~(P),~(W),~(C) and identities in~\cite[Lemma~2.5(c)]{Eick2000}, see Remark~\ref{rem:corcon}).   The generators in increasing index are
\(g_1 = t,\ g_2 = z,\ g_3 = b,\ g_4 = a\), with relative orders
\((2, M, N, N)\).   Define \((r'',s'')\in\Z_N^2\) by \((r'',s'') := (r,s)
Z^{-2c}\) (row-vector convention), so that the base part of
\(z^t = z^{1-2c}\beta_0\) from \((\mathcal{I}_2')\) is
\(\beta_0 = a^{-r''}b^{-s''}\).   The conjugation rewriting rules are
\begin{equation}\label{eq:rules-entries}
	\begin{aligned}
		g_2^{g_1} &= z^t = g_1^0\,g_2^{1-2c}\,g_3^{-s''}\,g_4^{-r''},\\
		g_3^{g_1} &= b^t = g_4,\quad g_4^{g_1} = a^t = g_3,\\
		g_3^{g_2} &= b^z = g_3^{\delta}\,g_4^{\gamma},\quad
		g_4^{g_2} = a^z = g_3^{\beta}\,g_4^{\alpha},\\
		g_4^{g_3} &= a^b = g_4.
	\end{aligned}
\end{equation}

\paragraph{(P) Identities \(g_j^{g_i^{e_i}} = g_j\).}
The pairs \((1,3),(1,4),(3,4)\) are immediate from \(a^{t^2}=a\), \(b^{t^2}=b\), and
\(a^{b^N}=a\), since \(t\) swaps \(a\) and \(b\), and \([a,b]=1\) .   For \((2,3)\) and \((2,4)\), we have that the powers \(b^{z^M}\) and \(a^{z^M}\) in
normal form are read from the second and first rows of \(Z^M\)
respectively; by~\eqref{eq:Zj-entries} at \(j=M\),
\[
b^{z^M} = a^{A_M\gamma}\,b^{A_M\delta + B_M},
\qquad
a^{z^M} = a^{A_M\alpha + B_M}\,b^{A_M\beta},
\]
and~\eqref{eq:N1} forces both to equal \(b\) and \(a\).
For \((1,2)\), \(z^{t^2} = z\) is given by Proposition~\ref{prop:zt}, by~\eqref{eq:N3} and by~\eqref{eq:N4}.

\paragraph{(W) Identities \(\bigl(g_j^{g_i}\bigr)^{e_j} = \bigl(g_j^{e_j}\bigr)^{g_i}\), for \(i<j\).}
As in the proof of Theorem~\ref{thm:main}, the non-trivial pair is \((1,2)\). We must check \((z^t)^M = 1\).   From
\(z^t = z^k\beta_0\) with \(k=1-2c\), expanding inside the semidirect
product as in~\eqref{eq:collect}, we have
\[
(z^k\beta_0)^M
\;=\; z^{kM}\,\beta_0^{ S_M(Z^k)}.
\]
Since \(z^M = 1\) we have \(z^{kM} = 1\), so the relation
\((z^t)^M = 1\) holds if and only if \(\beta_0^{S_M(Z^k)} = 1\) in \(B\),
i.e.\ \((r,s) Z^{-2c} S_M(Z^k)\equiv(0,0)\pmod N\);   this is
precisely condition~\eqref{eq:N5}. 
Pairs
\((1,3),(1,4),(2,3),(2,4)\), \((3,4)\) are handled identically to the
symmetric case.

\paragraph{(C) Identities \((g_k^{g_j})^{g_i} = g_k^{g_j g_i}\), \(k>j>i\).}
The triples \((1,3,4)\) and \((2,3,4)\) are unaffected by the form of
\(Z\), since \([a,b]=1\).   The critical triples are \((1,2,3)\) and
\((1,2,4)\). We must check that the two reductions of \(g_3 g_2 g_1\)
and \(g_4 g_2 g_1\) to normal form agree.

We carry out the computation for \((1,2,4)\); the case \((1,2,3)\) is
analogous.   The two reduction paths express \(a^{zt}\) in normal form;
since \(a^{zt}\) is a single group element, the two paths must
a fortiori produce the same word.

\emph{Path I (use \(g_2 g_1 = g_1 g_2^{1-2c}g_3^{-s''}g_4^{-r''}\) first).}
We have
\(a^{zt} = a^{t\,z^{1-2c}a^{-r''}b^{-s''}}\). 
Since \(a^t=b\) and conjugation by an element of \(B\) is trivial on \(B\),
this gives \(a^{zt} = b^{z^{1-2c}} = b^{z^k}\).   By~\eqref{eq:Zj-entries},
\[
b^{z^k} = a^{A_k\gamma}\,b^{A_k\delta+B_k}.
\]

\emph{Path II (use \(g_4^{g_2} = g_3^\beta g_4^\alpha\) first).}
We have
\( (a^z)^t = (a^\alpha b^\beta)^t = b^\alpha a^\beta = a^\beta b^\alpha\).

The two paths produce the same group element by associativity, so the
two normal forms must agree
\begin{equation}\label{eq:C124}
	A_k\gamma \equiv \beta\pmod N,
	\qquad
	A_k\delta + B_k \equiv \alpha \pmod N.
\end{equation}
These are two of the four congruences in~\eqref{eq:N2}.   The analogous
computation for the triple \((1,2,3)\), starting from
\(b^{zt}\), yields the other two congruences in~\eqref{eq:N2}:
\begin{equation}\label{eq:C123}
	A_k\beta \equiv \gamma,
	\qquad
	A_k\alpha + B_k \equiv \delta \pmod N.
\end{equation}

Hence~\eqref{eq:N2} is precisely the (C) consistency requirement on
the matrix~\(Z\).   Together with~\eqref{eq:N1} (giving (P) and (W)),
\eqref{eq:N3} and~\eqref{eq:N4} (giving the \((1,2)\) entries of (P)), and ~\eqref{eq:N5}
(giving the \((1,2)\) entries of  (W)), all consistency identities are satisfied, and the
PC-presentation defines a polycyclic group of order
\(\prod_i\!e_i = 2\,M\,N\,N = 2N^2M\).

It remains to verify the structural conditions. In particular, we have to show  that the
subgroups \(H = \langle a,b,t\rangle\) and \(K = \langle z\rangle\) are
of orders \(2N^2\) and \(M\) respectively, that \(H\cap K = 1\), and that
\(B = \langle a,b\rangle\trianglelefteq G\).   These follow
exactly as in the proof of Theorem~\ref{thm:main}, from the uniqueness
of the normal form \(t^{d_1}z^{d_2}b^{d_3}a^{d_4}\).   The subgroup
\(H = \langle a,b,t\rangle\) consists of the elements whose normal form
has \(d_2 = 0\) (no \(z\)), so \(|H| = 2N^2\); with the relations
\(a^N = b^N = t^2 = 1\), \([a,b]=1\), \(a^t = b\), \(b^t = a\) this gives
\(H\cong\Wr_{2^n}\).   The subgroup \(K = \langle z\rangle\) has order
\(M\), so \(K\cong C_{2^m}\).   Their intersection \(H\cap K\) has both
\(d_2 = 0\) (from \(H\)) and \(d_1 = d_3 = d_4 = 0\) (from \(K\)), hence is
trivial, and \(|HK| = |H|\,|K| = 2N^2M = |G|\).   Finally the conjugation
rules \(a^z = a^{\alpha}b^{\beta}\) and \(b^z = a^{\gamma}b^{\delta}\) lie
in \(B\), so \(B^z = B\); together with \(B^t = B\) (the swap
\(a^t = b\), \(b^t = a\)) this shows that \(B = \langle a,b\rangle\) is
normalised by every generator, hence \(B\trianglelefteq G\).
\hfill\(\square\)

\section{The exceptional modulus \(M = 4\)}\label{sec:M4}

At \(M = 4\) Theorems~\ref{thm:main} and~\ref{thm:nonsym} still hold in
entry form, but their conditions degenerate and it is necessary to reformulate the eigenvalue congruences.   The integer \(k = 1-2c\) satisfies
\(k\equiv -1\pmod 4\) for both odd residues \(c = 1, 3\), and \(k\equiv 1\pmod 4\) for both even residues \(c = 0, 2\), so the
involution condition~\eqref{eq:N3} (equivalently~\eqref{eq:W4}) holds
vacuously.   Moreover only \(2c\bmod 4\) -- equivalently the parity of
\(c\) -- enters the relations, so the two odd values \(c = 1\) and
\(c = 3\) yield the same presentation and there is a single odd-\(c\)
case (we take \(c = 1\)).   The remaining conditions do not vanish but
degenerate. With \(k\equiv -1 \pmod 4\) the kernel condition~\eqref{eq:W5}
collapses to the linear \((Z-J)\binom rs\equiv 0\pmod N\),
condition~\eqref{eq:W6} becomes the matrix identity \(Z^2\equiv I\pmod N\),
and~\eqref{eq:W7} becomes \((I+Z)\binom rs\equiv 0\pmod{N/2}\) -- all
still necessary, and 
sufficient (Theorems~\ref{thm:nonsym-M4} and~\ref{thm:M4}).   The exceptionality occurs precisely in the case  \(M=4\), indeed
this case is the unique modulus at which the non-abelian quotient
\(G/B\) of order \(2M = 8\) coincides with \(D_8\), the only non-abelian
\(2\)-group of that order with a cyclic subgroup of index \(2\); for \(M\ge 8\) the three involutions of \(\Z_M\)
distinguish the three types and Theorems~\ref{thm:main}
and~\ref{thm:nonsym} apply.

We exhibit what survives at \(M = 4\).   In entry form the conditions
\eqref{eq:N1}--\eqref{eq:N5} are necessary and sufficient
(Theorem~\ref{thm:nonsym-M4}); in the symmetric case they reduce to the
three congruences of Theorem~\ref{thm:M4}.   The one feature special to
\(M = 4\) is that the matrix identity \(Z^2 = I \pmod N\) admits no congruence on
\(\sigma,\mu\) modulo \(N\), in particular,  in eigenvalue coordinates it requires the
modulus-\(2N\) refinement~\eqref{eq:M4eig}.

\begin{theorem}
	\label{thm:nonsym-M4}
	Let \(N = 2^n\) and \(M = 4\).   The
	presentation~\eqref{eq:pres-entries} defines an exact product
	\(G = HK\) with \(H\cong\Wr_{2^n}\), \(K\cong C_4\) and \(B\) normal in
	\(G\) \emph{if and only if} the tuple
	\((\alpha,\beta,\gamma,\delta,r,s,c)\) satisfies
	conditions~\eqref{eq:N1}, \eqref{eq:N2}, \eqref{eq:N4}
	and~\eqref{eq:N5} of Theorem~\ref{thm:nonsym}; condition~\eqref{eq:N3}
	holds vacuously.
\end{theorem}

\begin{proof}
	This is the case \(M = 4\) of Theorem~\ref{thm:nonsym}; we check that
	neither direction uses \(M\ne 4\).
	
	\emph{Necessary condition.}   The identities \((\mathcal I_1')\)--\((\mathcal I_5')\)
	of the proof of necessary condition of Theorem~\ref{thm:nonsym} (Section~\ref{sec:nonsym}) extract
	\eqref{eq:N1}, \eqref{eq:N2}, \eqref{eq:N4} and~\eqref{eq:N5} from the
	relations \(z^M = 1\), \(z^t \in z^{k}B\), \((z^t)^t = z\) and
	\((z^t)^M = 1\), which hold in every exact product regardless of \(M\).
	Condition~\eqref{eq:N3} reduces to \((1-2c)^2\equiv 1\pmod 4\), true for
	every odd \(c\), hence vacuous.
	
	\emph{Sufficient condition.}   The sufficient condition of
	Theorem~\ref{thm:nonsym} -- the verification of the consistency
	conditions \textnormal{(P)}, \textnormal{(W)}, \textnormal{(C)} for the
	presentation~\eqref{eq:pres-entries} -- is uniform in \(M\). It uses
	\eqref{eq:N1} for the power identities of the pairs \((2,3),(2,4)\),
	\eqref{eq:N2} for the conjugation overlaps \((1,2,3),(1,2,4)\), and
	\eqref{eq:N4}, \eqref{eq:N5} for the pair \((1,2)\), invoking
	\eqref{eq:N3} only through its (here vacuous) instance.   Hence the
	presentation is consistent and defines a group of order \(2N^2M\), an
	exact product with \(B\) normal.   
\end{proof}
\begin{theorem}
	\label{thm:M4}
	Let \(N = 2^n\), \(M = 4\), and take \(Z\) symmetric circulant, with eigenvalues \(\sigma = 1+p+q\) and \(\mu = 1+p-q\).
	The presentation~\eqref{eq:main-pres} defines a non-split exact product
	\(G = HK\) with \(H\cong\Wr_{2^n}\), \(K\cong C_4\) and \(B\) normal
	\emph{if and only if}
	\begin{align}
		Z^2 &\equiv I\pmod N, \tag{M1}\label{eq:M4mat}\\
		(Z-J)\tbinom rs &\equiv\tbinom00\pmod N, \tag{M2}\label{eq:M4lin}\\
		(I+Z)\tbinom rs &\equiv\tbinom00\pmod{N/2}. \tag{M3}\label{eq:M4pow}
	\end{align}
	These are the \(M = 4\) congruences corresponding to
	\eqref{eq:W6}, \eqref{eq:W5} and \eqref{eq:W7}.   In eigenvalue
	coordinates the matrix identity~\eqref{eq:M4mat} is equivalent to
	\begin{equation}
		\sigma^2+\mu^2\equiv 2\pmod{2N}\quad\text{and}\quad
		\sigma^2\equiv\mu^2\pmod{2N},\tag{M1\('\)}\label{eq:M4eig}
	\end{equation}
	which is strictly stronger than \(\sigma^4\equiv\mu^4\equiv 1\pmod N\),
	the form to which the eigenvalue conditions of Theorem~\ref{thm:main}
	specialise at \(M = 4\).
\end{theorem}

\begin{proof}
\emph{Necessary condition.}   For \eqref{eq:M4mat}, the relation
\(a^{zt} = a^{t z^t}\) forces \(JZJ = Z^k = Z^{-1}\) (as
\(k = 1-2c\equiv -1 \pmod 4\)); since \(Z\) is symmetric circulant \(JZJ = Z\),
whence \(Z = Z^{-1}\), i.e.\ \(Z^2 \equiv I \pmod N\).   For \eqref{eq:M4lin}, the
closed form of Proposition~\ref{prop:zt} reads \(z^t = z^{-1}\beta\) with
\(\beta = a^{-r}b^{-s}\); applying \(t\),
\[
z = (z^t)^t = (z^{-1}\beta)^t = (z^t)^{-1}\beta^t
= (z^{-1}\beta)^{-1}\beta^t = \beta^{-1}\,z\,\beta^t ,
\]
and transporting \(\beta^{-1}\) past \(z\) by \(\beta z = z\,\beta^{Z}\)
gives  \(z=z\beta^{-Z}\beta^t\); additively this yields \((Z-J)\binom rs\equiv 0\pmod N\).
Finally, for \eqref{eq:M4pow}, the power consistency \((z^t)^4 = 1\)
gives \(S_4(Z^{-1})\binom rs\equiv 0\pmod N\); with \(Z^2 \equiv I \pmod N\) one has
\(S_4(Z^{-1}) = I+Z+I+Z = 2(I+Z)\), so \(2(I+Z)\binom rs\equiv 0\pmod N\).
	
	\emph{Eigenvalue form of~\eqref{eq:M4mat}.}   Write \(u = 1+p\), so
\(Z = \left(\begin{smallmatrix}u&q\\q&u\end{smallmatrix}\right)\) and
\(Z^2 = \left(\begin{smallmatrix}u^2+q^2 & 2uq\\ 2uq & u^2+q^2\end{smallmatrix}\right)\),
thus \(Z^2\equiv I\pmod N\) iff \(u^2+q^2\equiv 1\) and \(2uq\equiv 0\pmod N\).
The exact identities \(\sigma^2+\mu^2 = 2(u^2+q^2)\) and
\(\sigma^2-\mu^2 = 4uq\) turn these into \(\sigma^2+\mu^2\equiv 2 \pmod{2N}\) and
\(\sigma^2-\mu^2\equiv 0\pmod{2N}\), which is~\eqref{eq:M4eig}.   It forces
\(\sigma^4\equiv\mu^4\equiv 1\pmod N\) but not conversely.
	
	\emph{Sufficient condition.}   Conditions \eqref{eq:M4mat}--\eqref{eq:M4pow} are
the \(M = 4\) instances of \eqref{eq:W6}, \eqref{eq:W5} and
\eqref{eq:W7}; together with \eqref{eq:W1}--\eqref{eq:W3} -- trivial
here, since \eqref{eq:M4mat} gives \(\sigma^2\equiv\mu^2\equiv 1\pmod N\)
and hence \(\sigma^4\equiv\mu^4\equiv 1\pmod{N}\) and \(\sigma^4\equiv\mu^4\equiv 1\pmod{2N}\)
-- and \eqref{eq:W4} (vacuous), they form the
full hypothesis of Theorem~\ref{thm:main} at \(M = 4\), whose
sufficiency proof (Section~\ref{sec:suff}) is uniform in \(M\).  
\end{proof}


\section{Open problems}\label{sec:conclusion}


In this paper, we classify the
Zappa--Sz\'ep products \(G = HK\), where \(H = C_{2^n}\wr C_2\) is the wreathed
\(2\)-group and \(K\) a cyclic group of order \(2^m\), in which the base
\(B \cong C_{2^n}\times C_{2^n}\) of \(H\) is normal in \(G\).

It is natural to ask whether \(B\) is normal in every exact
product \(G = HK\) with \(H\cong\Wr_{2^n}\) and \(K\cong C_{2^m}\).   However, the following counterexample shows that this conjecture is false in general.  Among the exact products
\(\Wr_4\cdot C_4\) \textnormal{(}order \(128\)\textnormal{)} there are
exactly \(23\) in which \(B\) is not normal, the first being
\(\texttt{SmallGroup}(128,521)\), in the \textsf{GAP} small-groups
library; among the exact products \(\Wr_4\cdot C_8\)
\textnormal{(}order \(256\)\textnormal{)} there are again many, for
instance \(\texttt{SmallGroup}(256,2992)\).   The normality hypothesis   on \(B\)
 is therefore an effective  restriction, and the present
classification -- Theorems~\ref{thm:main} and~\ref{thm:nonsym} --
describes precisely the \(B\)-normal exact products.

Some directions remain open.   First, and most prominently, as already said the
exact products in which \(B\) is not normal fall outside the
present classification and call for a separate treatment; by the
experiments performed by \textsf{GAP},  they could
amount to the transitive \(2\)-groups on \(2^m\) points arising as
\(G/\operatorname{core}_G(H)\), together with the attendant extension
data.

A more ambitious direction is the classification of exact
products of two wreathed groups, where both factors have a base of rank
two and the analysis is governed by a pair of commuting swap operators.


\end{document}